\documentclass{amsart}
\usepackage{amsmath}
\usepackage{amsthm}
\usepackage{amsfonts}
\usepackage{amssymb}
\usepackage[all]{xy}
\usepackage{color}
\usepackage{tikz}
\usetikzlibrary{matrix}
\usepackage{verbatim}
\usepackage{mabliautoref}
\usepackage{tikz-cd}

\usepackage{enumitem}

\setlist[itemize,enumerate]{leftmargin=20pt}

\usepackage[left=1.02in,top=1.0in,right=1.02in,bottom=1.0in]{geometry}

\author{Zsolt Patakfalvi}
\address{
EPFL\\
SB MATHGEOM CAG  \\
MA C3 635 (B\^atiment MA) \\
Station 8 \\
CH-1015 Lausanne}
\email{zsolt.patakfalvi@epfl.ch}
\author{Joe Waldron}
\address{
Former:\newline
EPFL\\
SB MATHGEOM CAG  \\
MA C3 615 (B\^atiment MA) \\
Station 8 \\
CH-1015 Lausanne}
\address{
	Current:\newline
	Princeton University\\
	Fine Hall\\
	Washington Road \\
	NJ 08544 \\
	USA}
\email{jaw8@math.princeton.edu }

\newcommand{\PP}{\mathbb{P}}

\newcommand{\lin}{\sim}
\newcommand{\isom}{\cong}
\newcommand{\num}{\equiv}

\newcommand{\bF}{\mathbb{F}}
\newcommand{\bQ}{\mathbb{Q}}
\newcommand{\bR}{\mathbb{R}}
\newcommand{\bZ}{\mathbb{Z}}

\newcommand{\sA}{\mathcal{A}}

\newcommand{\sF}{\mathcal{F}}
\newcommand{\sG}{\mathcal{G}}
\newcommand{\sH}{\mathcal{H}}
\newcommand{\sL}{\mathcal{L}}
\newcommand{\sM}{\mathcal{M}}
\newcommand{\sO}{\mathcal{O}}
\newcommand{\sT}{\mathcal{T}}

\newcommand{\sHom}{{\mathcal H }{\it om}}

\newcommand{\im}{\mathrm{im}}

\newcommand{\sC}{\mathcal{C}}

\newcommand{\sX}{\mathcal{X}}
\newcommand{\sY}{\mathcal{Y}}
\newcommand{\sZ}{\mathcal{Z}}

\newcommand{\Q}{\mathbb{Q}}

\newcommand{\ov}{\overline{v}}

\DeclareMathOperator{\red}{red}
\DeclareMathOperator{\codim}{codim}

\DeclareMathOperator{\spec}{Spec}
\DeclareMathOperator{\reg}{reg}

\DeclareMathOperator{\Pic}{Pic}
\DeclareMathOperator{\Fr}{Fr}
\DeclareMathOperator{\Det}{det}
\DeclareMathOperator{\rk}{rk}
\DeclareMathOperator{\Der}{Der}
\DeclareMathOperator{\Hom}{Hom}
\DeclareMathOperator{\Spec}{Spec}

\DeclareMathOperator{\Supp}{Supp}
\DeclareMathOperator{\Frac}{Frac}

\newcommand{\Nagataone}{\textrm{N-1 }}

\begin{document}

\title{Singularities of general fibers and the LMMP}

\begin{abstract}
We use the theory of foliations to study the relative canonical divisor of a normalized inseparable base-change.   
Our main technical theorem states that it is linearly equivalent to a divisor with positive integer coefficients divisible by $p-1$.  We deduce many consequences about the fibrations of the minimal model program: for example the general fibers of terminal Mori fiber spaces of relative dimension $2$ are normal in characteristic $p\geq 5$ and smooth in characteristic $p\geq 11$.
\end{abstract}

\maketitle

\section{Introduction}

In characteristic zero, the general fiber of a fibration 
between smooth varieties is smooth, which no longer holds in positive characteristic. Indeed,   the general fiber may be non-normal or even non-reduced.  It could be hoped that if one has sufficient control of the invariants of the varieties involved, it might be possible to bound this behavior to small primes.  An example of this is the fact that quasi-elliptic surfaces exist only in characteristics $2$ and $3$.  This can be seen as a consequence of Tate's genus change formula, which implies that if $X$ is a regular geometrically reduced curve over $k$ and $Y$ is the normalization of the base change by Frobenius of $k$ then the degree of $\omega_{Y/X}$ is divisible by $p-1$ (see \cite{tate_genus_change}, or \cite{schroer_tate} for a modern interpretation).  

We are interested in generalizing this statement to higher dimensions and to the geometrically non-reduced case. Our main motivation is to restrict bad behavior to small primes for the fibrations of the Minimal Model Program, such as Mori fiber spaces and the Iitaka fibration,  as this phenomenon seems to be a cornerstone of the Minimal Model Program in positive characteristic.  

First, we state our general result, which is the generalization of the aforementioned Tate's genus change formula. Tanaka has recently studied the effect of Frobenius base change on canonical divisors in \cite{tanaka_behaviour}, essentially proving that $\omega_{Y/X}$ is effective in arbitrary dimensions.  We extend this by proving that it can be chosen so that its coefficients are non-negative integers divisible by $p-1$.  In particular, our result covers the case where $X$ is not geometrically reduced.  Note that the theorem as stated below is about varieties over non-closed fields. To apply it to  fibrations $f : \sX \to T$ over algebraically closed fields one simply applies it to the generic fiber of the fibration, which is a variety over the non-closed field $K(T)$.

\begin{theorem}\label{thm:generic}
Let $X$ be a normal scheme of finite type over a field $k$ of characteristic $p>0$, such that $k$ is algebraically closed in $K(X)$.  Let $k'/k$ be a field extension and $Y$ be the normalization of the reduced subscheme of $X\otimes_k k'$.  Then there is an effective Weil divisor $C$ on $Y$ such that $K_Y+(p-1)C\lin \phi^*K_X$, where $\phi:Y\to X$ is the natural map.   
\end{theorem}

In the case where $X\otimes_k k'$ is reduced, it turns out that our $C$ is determined canonically, and while its method of construction is different, $(p-1)C$ turns out to be equal to the usual conductor of the normalization.

\begin{theorem}\label{thm:conductor}
	Under the notations and assumptions of \autoref{thm:generic}, if in addition $X\otimes_k k'$ is reduced and $X$ is proper over $k$, then $C$ can be chosen so that $(p-1)C$ is equal to the divisorial part of the usual conductor scheme.  In other words, the coefficients of the conductor divisor are divisible by $p-1$.
\end{theorem}


The proof of \autoref{thm:generic} and \autoref{thm:conductor} is geometric, that is, we reduce it to fibrations via taking adequate subrings of $k$ and $k'$ that are finitely generated as algebras over $\bF_p$. Then, the main idea is to view the base change and the normalization of the corresponding fibration as a quotient by a foliation.  The most important step is then to show that the Weil divisor which appears in the canonical bundle formula for quotient by a foliation can be taken to be effective in our situation.   

As a result, we obtain the following by restricting to the general fiber:

\begin{corollary}
\label{cor:singularities_of_general_fiber}
Let $\sX$ be a normal variety over a perfect field $k$ of characteristic $p>0$, and let $f : \sX \to T$ be a proper morphism to a normal variety over $k$, whose generic fiber is geometrically reduced. Then, the divisorial part of the conductor on the normalization of a general fiber $\sX_t$ has positive integer coefficients divisible by $p-1$.

\end{corollary}

Next we deduce some geometric consequences of \autoref{thm:generic}, bounding the existence of bad fibrations. We start with a result which applies to arbitrary dimensional varieties.

\begin{theorem}\label{cor:general_dimension}
If $f : \sX \to T$ is a proper fibration between normal varieties over a field $k$ of characteristic $p>0$,  $f_*\sO_{\sX}=\sO_T$, such that a general fiber of $f$ is not normal and one of the following holds:
\begin{enumerate}
\item $K_{\sX} \equiv_f 0$ or 
\item \label{itm:smooth:pseff} $\sX$ is $\bQ$-factorial, $-K_{\sX}$ is ample$/T$ and $\rho(\sX/T)=1$
\end{enumerate}
then the normalization $Y$ of the maximal reduced subscheme of the general fiber cannot be smooth if $p> p_0$ where
$$
p_0=
\begin{cases}
\dim(X)-\dim(T)+2, & \text{in case (a)} \\
\dim(X)-\dim(T)+1, & \text{in case (b)}
\end{cases}
$$

In particular, the above statements pertain to the Iitaka fibration and to Mori fiber spaces respectively. 
\end{theorem}

We can say more for fibrations of relative dimension $2$, in particular for Mori fiber spaces with terminal singularities:

\begin{theorem}\label{cor:normality_short}
Let $X$ be a normal, projective and Gorenstein surface over a field $k$ of characteristic $p>0$, with $k$ algebraically closed in $K(X)$, and such that $-K_{X}$ is ample.  If $p>3$ then $X$ is geometrically normal.
\end{theorem}






Further details about the possibilities for the normalized base change are given before the proof as \autoref{cor:normality}.
Combining \autoref{cor:normality_short} with arguments from \cite[5.1]{hirokado} and \cite{schroer} we also deduce generic smoothness for terminal Fano fibrations in sufficiently large characteristics:

\begin{corollary}\label{cor:generic_smoothness}
Let $f:\sX\to Z$ be a projective fibration of relative dimension $2$ with $f_*\sO_{\sX}=\sO_Z$  over a perfect field of characteristic $p\geq 11$, such that $-K_{\sX}$ is ample over $Z$.  Then a general fiber of $f$ is smooth.
\end{corollary}

It then follows from generic smoothness and \cite[4.8]{rational_connectedness_GFR} that if $C$ is a rational curve, $\sX$ is separably rationally connected.

\begin{corollary}\label{cor:separably_rationally_connected}
Let $f:\sX\to \mathbb{P}_k^1$ be a projective fibration with $f_*\sO_{\sX}=\sO_{\mathbb{P}^1_k}$ from a terminal threefold over a perfect field of characteristic $p\geq 11$, such that $-K_{\sX}$ is ample over $\mathbb{P}_k^1$.  Then $\sX$ is separably rationally connected.
\end{corollary}

A special case of \autoref{cor:general_dimension} which is worth mentioning in its own right, is the analogue of quasi-elliptic surfaces with higher dimensional base.  This also appeared recently in \cite{zhang_abundance}.

\begin{corollary}\label{cor:quasi-elliptic}
Let $f:\sX\to B$ be a proper fibration with $f_*\sO_{\sX}=\sO_T$, generically of relative dimension $1$ between normal varieties, over a field $k$ of characteristic $p\geq 5$, such that $K_{\sX}$ is numerically trivial over a non-empty open subset of $B$.  Then a general fiber is a smooth curve of genus $1$.
\end{corollary} 

As an application of our geometric results, we study Kodaira vanishing for terminal Mori fiber spaces.  Examples of del Pezzo surfaces over imperfect fields for which Kodaira vanishing fails are known to exist in characteristic $2$, for example \cite{maddock_del_pezzo} and \cite{schroer_weak_del_pezzo}.
We prove that such surfaces do not exist in characteristics at least $5$, as a special case of the following:  

\begin{theorem}\label{thm:generic_kodaira}
Let $X$ be a normal projective variety over a field $k$ of characteristic $p$, such that either
\begin{enumerate}
 \item \label{itm:generic_kodaira:ample} $-K_X$ is an ample $\bQ$-Cartier divisor and  $p>2 \dim X$, or
 \item \label{itm:generic_kodaira:nef} $-K_X$ is a nef $\bQ$-Cartier divisor and $p>2 \dim X + 1$. 
\end{enumerate}
Let $L$ be an ample Cartier divisor on $X$.  Then $H^1(X,\sO_X(-L))=0$.  
\end{theorem}

We then use this together with our previous results to prove that  smooth threefolds admitting a Mori fiber space structure of relative dimension $2$ in characteristic $p\geq 5$  satisfy Kodaira vanishing for $H^2$.  We note that here we use the usual definition, which is that $f : X \to T$ is a smooth Mori fiber space if $X$ is smooth, $-K_X$ is $f$-ample, and $\rho(X/T)=1$.

\begin{theorem}\label{cor:kodaira_gorenstein}
Let $f : X \to T$ be a smooth projective $3$-fold admitting a Mori fiber space to a smooth curve over an algebraically closed field of characteristic $p\geq 5$ and let $L$ be an ample Cartier divisor on $X$. Then $H^2(X,\sO_X(K_X + L))=0$.

Furthermore if $p\geq 11$, then the following isomorphism of $k$-vector spaces holds:
$$H^1(X,\sO_X(K_X+L))\isom \bigoplus_{\parbox{75pt}{\begin{center}\tiny $t \in T$ closed point, \\[1pt]
$\left(X_t \right)_{\red}$ is non-normal, \\[1pt]
$H^1\left( \left. X_t, - L \right|_{X_t} \right) \neq 0$\end{center}}} \left(R^1 f_* \sO_X (K_X +L)\right)_t,$$
where $\left(R^1 f_* \sO_X (K_X +L)\right)_t$ denotes the stalk (not the fiber) of the corresponding sheaf, which stalk is Artinian and non-zero.
\end{theorem}

\begin{remark}
The moral of the part of \autoref{cor:kodaira_gorenstein} pertaining to $H^1(X, K_X +L)$ is that Kodaira vanishing holds (for the fixed line bundle $L$) in this cohomological degree if and only if it holds for all fibers which have non-normal reduced structures. In characteristic zero the latter vanishing is somehow forced due to global reasons. It is not clear whether or not any bad behavior can happen in  high enough characteristics. 	\end{remark}

\subsection{Organization of the paper}

In \autoref{sec:prelim} we review some necessary background. In \autoref{sec:can_bdl} we prove our main technical statements (\autoref{thm:foliation} and \autoref{thm:p-1C_conductor}), which is the canonical bundle formula for inseparable base-changes reformulated for geometric fibrations. Then, in \autoref{sec:geom_conseq}, we establish the geometric consequences and in \autoref{sec:irregularity} the Kodaira vanishing statement.

\subsection*{Acknowledgments}

We would like to thank Gavin Brown for asking the question about normality and reducedness of fibers of Mori fiber spaces which partially motivated the project, Andrea Fanelli, Omprokash Das and Johan de Jong for the useful conversations, and Hiromu Tanaka for useful comments and for pointing out \cite{gerstenhaber}. 

\section{Preliminaries}
\label{sec:prelim}

A \emph{variety} over a field $k$ is a separated, integral scheme of finite type over $k$. \emph{Fibration} in this article means a surjective map of varieties over some field $k$.  Our fields may be imperfect unless otherwise stated.

\subsection{Generic and general fibers}

Let $f:\sX\to B$ be a morphism of varieties over a field $k$.  The \emph{generic fiber} $X$ of $f$ is the scheme $\sX\times_B \Spec(k(\eta))$ where $\eta$ is the generic point of $B$.  This is a scheme over $K(B)$, the function field of $B$.  The \emph{geometric generic fiber} of $f$ is the scheme $X\times_{K(B)}\overline{K(B)}$.  By saying that a general fiber of $f$ satisfies a certain property, we mean that there is a non-empty open subset $U\subset B$ such that the scheme theoretic fiber over every closed point of $U$ satisfies that property.

In general, if $k$ is perfect, the properties of the geometric generic fiber reflect those of a general fiber (as stated in \autoref{prop:fibers}) while the properties of the generic fiber reflect those of $\sX$.

\begin{proposition}\label{prop:fibers}
Let $f:\sX\to B$ be a morphism of varieties over a perfect field $k$.  Then the geometric generic fiber is normal (resp. regular, reduced) if and only if a general fiber is normal (resp. regular, reduced). 
\end{proposition}

\begin{proof}
Although the statement is well known to experts, we list some references and facts from which it can be put together. The claim that geometric normality (resp. geometric regularity and geometric reducedness) are open properties (in $B$) is shown in  \cite[Thm 6.9.1]{EGA_IV_II}, \cite[Thm 12.1.1]{EGA_IV_III}. Then it follows that these properties hold for the geometric generic fiber if and only if they hold for fibers over a non-empty open set. To conclude,  one has to show that for perfect ground fields the geometric and the non-geometric versions of these three notions coincide. For reduced this is \cite[Exc 3.15.b]{hartshorne_ag}. 
For normal it is easiest is to think about it as $S_2 + R_1$ noting that $S_n$ is invariant under base extension (e.g., \cite[Prop 5.3.1]{EGA_IV_II}) and  for finite type schemes over perfect fields smoothness is equivalent to regularity (e.g. \cite[Thm 22.5.8]{EGA_IV_I}). 
\end{proof}

\begin{lemma}\label{lem:fibre_normalisation}
Let $f:X\to T$ be a morphism of varieties and $\pi:Y\to X$ the normalization of $X$.  Suppose that for a general point $t\in T$, $Y_t$ is normal.  Then $Y_t$ is the normalization of $X_t$. 
\end{lemma}
\begin{proof}
This follows from the universal property of normalization:  If $Z\to X_t$ is a morphism from a normal variety, we get a morphism $Z\to X$, which therefore factors through $Z\to Y$ and hence through $Z\to Y_t$.
\end{proof}

\subsection{Reflexive sheaves}

Let $\sF$ be a reflexive coherent sheaf on a normal variety $X$, meaning that the natural map $\sF \to \sF^{**}$ is an isomorphism.  Denote by $\sF^{[n]}$ the $n^{\mathrm{th}}$ reflexive power of $\sF$, that is, $(\sF^{\otimes n})^{**}$.  We define the rank of $\sF$ to be the dimension of $\sF_{\eta}$ where $\eta$ the generic point of $X$.  We define the determinant of a rank $r$ reflexive sheaf $\sF$ to be $(\wedge^{r}\sF)^{**}$.  We take the double dual to ensure that the determinant is a rank $1$ reflexive sheaf, that is, it determines a class of Weil divisors.  For details on the properties of reflexive sheaves see \cite{hartshorne_reflexive}.

Let $X$ be a normal variety over a field $k$.  Then we set $\Omega_X:= \Omega_{X/k}$ to be the sheaf of K\"{a}hler differentials on $X$ over $k$, and define the tangent sheaf $\sT_X$ to be $\Omega_X^{*}$.  Then  $\sT_{X}$ can be also viewed the sheaf of derivations on $X$. More precisely, on an affine open set $U \subseteq X$ with $A:= \Gamma(U, \sO_X)$, $\sT_X$ corresponds to the $A$-module $\Hom_A(\Omega_{A/k}, A)$. Then
according to the universal property of K\"ahler differentials,  the latter is isomorphic to $\Der_k(A)$, so $\sT_X$ over $U$ corresponds to $\Der_k(A)$. 
Note, that if  $k$ is perfect, the regular locus $X_{\reg}$ is smooth over $k$ (e.g. \cite[Thm 22.5.8]{EGA_IV_I}), and hence $\sT_X|_{X_{\reg}}$ is locally free of rank $\dim X$.

An open set $\iota : U \hookrightarrow X$ is called big if $\codim_X (X \setminus U) \geq 2$.  A coherent sheaf is reflexive if and only if $\iota_* (\sF|_U) \cong \sF$ for a big open set $U$ via the natural morphism between the two \cite[Prop 1.6]{hartshorne_reflexive}. This in particular shows that saturated subsheaves of reflexive sheaves are reflexive (where saturated can be defined equivalently either by requiring the stalks to be saturated submodules or by requiring that sections are in the subsheaf if and only if they are in it at the generic point). Also, as $\sT_X$ is a reflexive sheaf by \cite[Cor 1.2]{hartshorne_reflexive}, if $k$ is perfect it is equal to the extension of the locally free sheaf of derivations  on  $X_{\reg}$ to $X$ via the above description of reflexive sheaves.

\subsection{Frobenius Morphisms}

In this subsection, let $X$ be a scheme over a field $k$ of characteristic $p>0$.

\begin{definition}
There exists an endomorphism denoted $\Fr_X:X\to X$ (or just $\Fr$ if there is no ambiguity) called the \emph{absolute Frobenius}, induced by the $p^{\mathrm{th}}$ power maps on rings.
\end{definition}



\begin{definition}
Let $X^{(-1)}:=X\otimes_kk$ where the right hand map is $\Fr_k$.  The induced $k$-morphism $F_{X/k}:X\to X^{(-1)}$ is called the \emph{relative Frobenius}.  We denoted it by $F$ if $X$ and $k$ are clear from the context.
\end{definition}

\begin{remark}
If $k$ is perfect then $X^{(-1)}$ is isomorphic to $X$ as a scheme, but not as a $k$-scheme.  In that case we can also denote the relative Frobenius by $F:X^{(1)}\to X$.   Unless otherwise specified, $F$ will be relative to the ground field.
\end{remark}

\subsection{Foliations}

In this subsection we describe a correspondence between certain finite purely inseparable morphisms and certain subsheaves of $\sT_{X}$.

\begin{definition}
We say that a purely inseparable morphism $f:X\to Y$ of characteristic $p$ schemes is of \emph{height 1} if there exists a morphism $g:Y\to X$ such that $g\circ f =\Fr$.
\end{definition}

\begin{definition}
Let $X$ be a normal variety over a perfect field of characteristic $p>0$.  By a \emph{foliation} on $X$ we mean a subsheaf $\sF\subset \sT_X$ which is saturated and closed under $p$-powers.   (Note that in characteristic $p$, if $d$ is a derivation then the composition of $p$ copies of $d$ as a function is also a derivation: this is what we mean by $p$-powers.)

\end{definition}

\begin{remark}
Usually in the above definition it is also required that $\sF$ is closed under Lie-brackets. However, this follows from the $p$-closedness assumption according to \cite{gerstenhaber}.
\end{remark}

\begin{proposition}
Let $X$ be a normal variety over a perfect field $k$ of characteristic $p>0$. Then, there is a $1$-to-$1$ correspondence as follows:
\begin{equation*}
\fbox{\parbox{220pt}{finite purely inseparable $k$-morphisms $f : X \to Y$ of height $1$ with $Y$ normal}}
\longleftrightarrow
\fbox{\parbox{80pt}{foliations $\sF \subseteq \sT_X$}}
\end{equation*}
This correspondence is given by:
\begin{itemize}
\item $\rightarrow$: $\sF$ is the sheaf of derivations that vanish on $\sO_Y \subseteq \sO_X$ ($X$ and $Y$ have the same topological space, so this containment does make sense), and 
\item $\leftarrow$: $Y= \Spec_X \sA$, where $\sA$ is the subsheaf of $\sO_X$ that is taken to zero by all the sections of $\sF$. 
\end{itemize} 
Furthermore, morphisms of degree $p^m$ correspond to foliations of rank $m$.

\end{proposition}

\begin{proof}
The correspondence is shown in the following diagram, where we comment below the diagram on the arrows. We denote by $\eta$ the generic point of $X$.
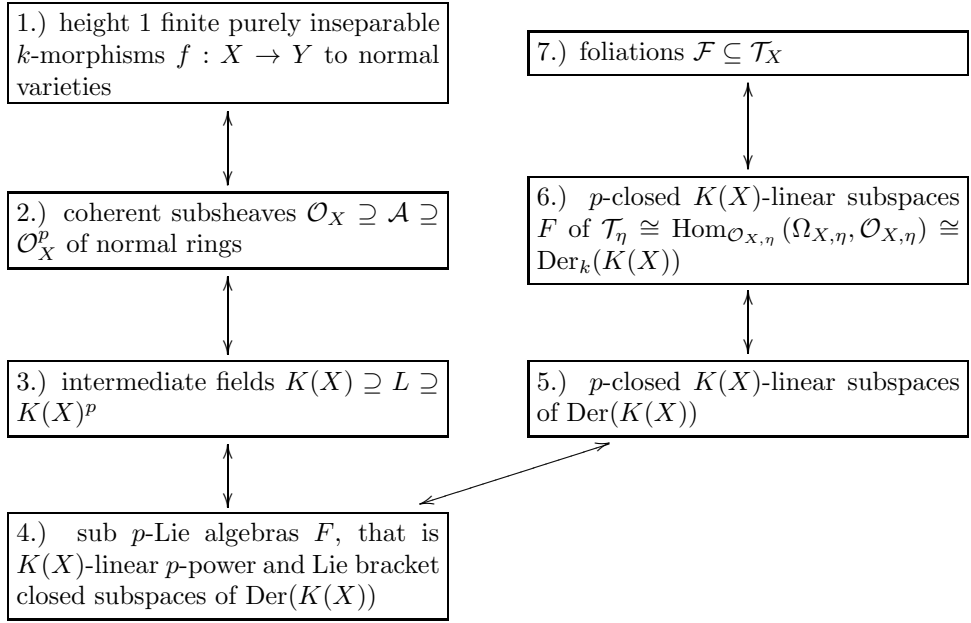
\begin{equation*}
\xymatrix{
\fbox{\parbox{160pt}{1.) height $1$ finite purely inseparable $k$-morphisms $f : X \to Y$ to normal varieties}} \ar@{<->}[d]
&
\fbox{\parbox{160pt}{7.) foliations $\sF \subseteq \sT_X$}}
\\
\fbox{\parbox{160pt}{2.) coherent subsheaves $\sO_X \supseteq \sA \supseteq \sO_X^p$ of normal rings}} \ar@{<->}[d]
&
\fbox{\parbox{160pt}{6.) $p$-closed $K(X)$-linear subspaces $F$ of $\sT_{\eta} \cong \Hom_{\sO_{X,\eta}}\left(\Omega_{X,\eta}, \sO_{X,\eta} \right) \cong \Der_k(K(X))$}} \ar@{<->}[u]
\\
\fbox{\parbox{160pt}{3.) intermediate fields  $K(X) \supseteq L \supseteq K(X)^p$}} \ar@{<->}[d]
&
\fbox{\parbox{160pt}{5.) $p$-closed $K(X)$-linear subspaces of $\Der(K(X))$}} \ar@{<->}[u]
\\
\fbox{\parbox{160pt}{4.) sub $p$-Lie algebras $F$, that is $K(X)$-linear $p$-power and Lie bracket closed subspaces  of $\Der(K(X))$}}\ar@{<->}[ur] 
}
\end{equation*}
\begin{itemize}
\item $1.) \leftrightarrow 2.)$: In the downwards direction, set $\sA:=\sO_Y$. Since $f$ is finite and purely inseparable, it is a homeomorphism, so $\sA$ lives on the same space as $\sO_X$ and $\sO_X^p$. In the other direction set $Y:= \Spec_X \sA$, with $k$-algebra structure induced by the composition $k\to \sO_X\hookrightarrow\sA$.
\item $2.) \leftrightarrow 3.)$: In one direction, we pass to the local rings at $\eta$. In the other direction, $\sA$ is the normalization of $\sO_X^p$ in $L$. This is an equivalence by the uniqueness of normalization.
\item $3.) \leftrightarrow 4.)$: This is the statement of \cite[first paragraph on page 189]{Jac75}.
\item $4.) \leftrightarrow 5.)$: This is proved in \cite{gerstenhaber}.  The statement is that being a sub Lie algebra is automatic for a $p$-closed subspace.
\item $5.) \leftrightarrow 6.)$: The only content here is that $\Der(K(X)) \cong \Der_k(K(X))$, that is an unconditional derivation $D$ of $K(X)$ is automatically zero on $k$. Indeed, it is automatically zero on $K(X)^p$, and since $k$ is perfect, $k \subseteq K(X)^p$. 
\item $6.) \leftrightarrow 7.)$: In one direction we set $F:= \sF_{\eta}$. In the other direction, we put a section $s$ of $\sT_X$ (over any open set) into $\sF$, if and only if $s_{\eta} \in F$. This yields the unique saturated subsheaf whose stalk at  $\eta$ is $F$. Since, $F$ is $p$-closed, so is $\sF$ by saturatedness.
\end{itemize}
Now we are left to contemplate that this correspondence is indeed given in both directions as claimed in the statement. By shrinking $X$ small enough, we may assume that $X$ is affine. Set $F:= \Gamma(X,\sF)$, $A:= \Gamma(X, \sO_X)$ and $B:= \Gamma(X, \sO_Y)$.  We see that our correspondence above agrees at the generic point $\eta$ with what we claimed in the statement. To show that it is true globally, we need to show two things: 
\begin{enumerate}
\item \label{goal:correspondence:two} $\{\partial \in \Der_k(A)| \forall f \in B: \partial(b) = 0\}$ \emph{is saturated.} As $\sO_X$ is torsion-free, we have: $\partial(b) = 0$ if and only if $\partial_\eta(b_{\eta})=0$, so $\partial$ being in the above subspace is decided at $\eta$, and hence, the above subspace of $\Der_k(A)$ is saturated.
\item \label{goal:correspondence:one} $C:=\{f \in A| \forall \partial \in F: \partial(f) =0\}$ \emph{is normal.} Let $a \in C_\eta$ be integral over $C$. As $A$ is integrally closed, $a \in A$ must hold. Also, as $a \in C_\eta$, we have  $\partial_\eta (a) =0$ for every $\partial \in F$. Then using that we are working in integral domains it follows that $\partial(a)=0$ also holds, and hence $a \in C$. 
\end{enumerate}

In order to see the last claim, that morphisms of degree $p^m$ correspond to foliations of rank $m$, it is enough to show that a field extension of degree $p^m$ corresponds to a $p$-closed subspace of dimension $m$.  Given a purely inseparable field extension $F_m/F_0$ of degree $p^m$, we can produce a tower $$F_m/F_{m-1}/.../F_0$$ where $F_i/F_0$ is of degree $p^i$.   Therefore if we prove the claim for $m=1$ we are done.  
But $F_1\isom F_0[x]/(x^p-\alpha)$ for some 
$\alpha\in F_0$, and so $\Der_{F_0}(F_1)=\langle\frac{\partial}{\partial x}\rangle$.

\end{proof}

\begin{proposition}
\label{prop:foliation_canonical}
Suppose $Y\to X$ is a finite purely inseparable $k$-morphism of height one between normal varieties over a perfect field $k$ of characteristic $p>0$, and $\sF$ the corresponding foliation.  Then
\begin{equation*}
 \omega_{Y/X}\cong (\det \sF)^{[p-1]}.
\end{equation*}
\end{proposition}

\begin{proof}
As both sides of the isomorphism in the statement are reflexive, we may restrict freely to any big open set. In particular, we may assume that both $X$ and $Y$ are smooth over $k$ and that $\sF \subseteq \sT_X$ is locally free.
Note that with these assumptions, both sides of the isomorphism become line bundles, and hence we may also assume that $k$ is algebraically closed.  
This is because two line bundles are isomorphic if and only if they are isomorphic after some field extension, and the canonical sheaves are compatible with base-extension. 
However, in this situation the statement was shown for example in \cite[Prop 1.1, (1.3)]{ekedahl_canonical}

\end{proof}

\subsection{Torsors and inseparable covers}

In this subsection we describe a way to produce purely inseparable covers which will be used in \autoref{sec:irregularity}.  Let $X$ be a scheme, and $\sL$ a line bundle on $X$.  We can consider $\sL$ as a group scheme over $X$ under addition.  In characteristic $p$ there is a morphism of group schemes $\sL\to \sL^p$ which is surjective in the fpqc topology, and which has some kernel $\alpha_{\sL}$.  Elements of $H^1(X,\alpha_{\sL})$ correspond to $\alpha_{\sL}$-torsors, which in our situation are purely inseparable covers of degree $p$. In summary:

\begin{theorem}[\cite{ekedahl_canonical}, \cite{maddock_del_pezzo}]\label{thm:ekedahl_cover}
Suppose $L$ is a line bundle on a normal variety $X$ over a field $k$ of characteristic $p$.  
Elements of the kernel of the Frobenius action
$$\Fr^*:H^1(X,\sL)\to H^1(X,\sL^p)$$ give rise to $\alpha_{\sL}$-torsors.
Furthermore, the torsor arising from any non-trivial element of the kernel is non-trivial, that is a purely inseparable cover of varieties of degree $p$.

 Finally, if $X$ is a projective, then any non-trivial $\alpha_{\sL}$-torsor $Z$ is also a projective, $G_1$ and $S_2$ variety which satisfies
$$\omega_Z\isom f^*(\omega_X\otimes \sL^{p-1}).$$
\end{theorem}
\begin{proof}
The corresponding statement in \cite{ekedahl_canonical} assumes $X$ is smooth in place of $G_1$, and \cite{maddock_del_pezzo} extends this to Gorenstein.  However, the construction of $Z$ does not depend on the singularities of $X$, and there exists a Zariski cover of $X$ by $U\isom\Spec(A)$ such that $f^{-1}(U)\isom \Spec(A[x]/(x^p-f))$ for some $f\in A$ with $f^p\notin A$. It follows that $Z$ is Gorenstein over Gorenstein points of $X$ and $S_2$ over $S_2$ points of $X$.  In the case of $S_2$ this follows from \cite[Proposition 5.2]{kollar_mori} for then it is enough that $A[x]/(x^p-f)$ is $S_2$ as an $A$ module, when in fact it is free.  The formula for $\omega_Z$ is obtained by taking reflexive closures of the corresponding formula which holds on the Gorenstein locus.
\end{proof}

\subsection{Conductor ideal}\label{subsec:conductor}

In this subsection we collect together some facts about the conductor ideal.  

\begin{definition}
A domain $R$ has the property \emph{\Nagataone} if its integral closure is a finite $R$-module.
\end{definition}

Let $R$ be an \Nagataone ring and $S$ its normalization. Let $\frac{a_1}{b_1},\dots, \frac{a_r}{b_r}$ be generators of $S$ as an $R$-module ($a_i,b_i \in R$). Define $b:= \prod_{i=1}^r b_i$. Then  $\Spec S \to \Spec R$ is an isomorphism over $D(b)$, and $Sb = \sum_{i=1}^r R a_i \frac{b}{b_i}$. Indeed, a general element of the left side and of the right side of the above equation are (for  $c_i \in R$): $\left( \sum_{i=1}^r c_i \frac{a_i}{b_i} \right) b=\sum_{i=1}^r c_i \left(a_i \frac{b}{b_i}\right)$. In particular, $Sb$ is a non-empty ideal of $S$, which is also an ideal of $R$. If $R$ is also Noetherian, then there exists a maximal such ideal, as the set of such ideals is closed under addition of ideals.

\begin{definition}\label{def:conductor}
Let $R$ be a Noetherian \Nagataone ring and $S$ its normalization.  The largest ideal of $R$ which is also an ideal of $S$ is called the \emph{conductor}.  By the above remarks, this exists.
\end{definition}

The following (well-known) lemma says that the divisorial part of the conductor appears in the canonical bundle formula for normalization.

\begin{lemma}
\label{lem:two_conductors}
Let $f : X \to Y$ be the normalization of an affine variety over a field $k$ such that $\omega_X \cong \sO_X$ and $\omega_Y \cong \sO_Y$, and set $S := \Gamma(X, \sO_X)$ and $R:= \Gamma(Y,\sO_Y)$. Then the  image of  the natural map $\omega_X \to f^* \omega_Y  $ (which is recalled in the proof) is equal to the conductor ideal, where  $f^* \omega_Y$ is identified with $\sO_X$.

\end{lemma}

\begin{proof}
First assume that $X$ and $Y$ are not necessarily affine, and $\omega_X$ and $\omega_Y$ are assumed to be locally free.  We provide then the definition of the natural map $\omega_X \to f^* \omega_Y  $ compatible with restriction to open sets. Let $ \xi : f_* \omega_X \to \omega_Y \to (f_* \sO_X) \otimes_{\sO_Y} \omega_Y$ be the composition of the trace map with the natural inclusion. As both the trace map and the inclusion are compatible with restriction to open sets, indeed, $\xi$ is compatible with restrictions to open sets. However, a priori $\xi$ is only an $\sO_Y$-homomorphism, and we need to argue that it is also an $f_* \sO_X$-homomorphism. However, as $f$ is an isomorphism over an open set $U$ of $Y$, $\xi$ is an isomorphism over $U$, and hence $\xi$ an $f_* \sO_X$-homomorphism over $U$. As all the sheaves involved are torsion-free, it follows then that $\xi$ is also an $f_* \sO_X$-homomorphism.

Consider now the special affine situation of the lemma. The map $\xi$ can then be identified with the map $\zeta: \Hom_R(S,R) \to S$ defined by $\phi \mapsto \phi(1)$. It is enough to prove then that $\im \zeta = I$, where $I$ is the largest ideal of $R$ which is also an ideal of $S$. 

For $s \in S$, denote by $m_s \in \Hom_S(S,S)$ the multiplication by $s$. If $s \in I$, then $m_s \in \Hom_R(S,R)$. As $m_s(1) = s$, it follows that $I \subseteq \im \zeta$. 

For the other direction we have to show that $\im \zeta$ is an ideal also in $S$. However, for this we only have to show that $\zeta$ is $S$-linear which we have already seen, as $\zeta$ can be identified with $\xi$ and $S$ with $f_* \sO_X$.
\end{proof}

\subsection{N\'{e}ron-Severi groups}

\begin{definition}
	Let $f:X\to S$ be a morphism of normal varieties.  Then $N_1(X/S)$ is the $\bR$-vector space of $1$-cycles modulo numerical equivalence with respect to all Cartier divisors on $X$.
	
	$\rho(X/S)$ is the rank of $N_1(X/S)$.
	\end{definition}

\begin{lemma}\label{lem:base_change_rho}
	Let $f:X\to S$ be a morphism of normal varieties, and let $\phi:T\to S$ be a purely inseparable morphism.  Let $Y$ be a normal variety together with a purely inseparable morphism $\psi:T\to X$ which fits into a diagram

\begin{center}	
	\begin{tikzcd}
		Y\ar{r}{\psi}\ar{d}{g} & X\ar{d}{f}\\
		T\ar{r}{\phi} & S
		\end{tikzcd}
\end{center}
	
	Then $\rho(X/S)=\rho(Y/T)$.
	\end{lemma}
\begin{proof}
	We can extend the diagram as follows, where top and bottom compositions are (global) Frobenius.

\begin{center}	
		\begin{tikzcd}
		Y\ar{r}{\psi}\ar{d}{g} & X\ar{d}{f}\ar{r} & Y\ar{d}\\
		T\ar{r}{\phi} & S\ar{r}& T
	\end{tikzcd}
\end{center}

Pushforward of $1$-cycles via the universal homeomorphism composition yields a bijective linear map $$N_1(X/S)\to N_1(Y/T)\to N_1(X/S)$$ as required.

\end{proof}

\section{The canonical bundle formula}
\label{sec:can_bdl}

\subsection{The main theorem}
\label{subsec:thm_main}

The following theorem will imply \autoref{thm:generic} by restriction to the generic fiber of a suitably constructed fibration.  We reformulate the theorem in a geometric way in order to be able to work over a perfect field and use the theory of foliations. The deduction of \autoref{thm:generic} from \autoref{thm:foliation} will be given later in this subsection. 

As we construct $C$ via the theory of foliations, in order to prove  \autoref{thm:conductor}, we need to relate our divisor to the usual conductor.  Here we will show first that the effective divisor $C$ we find via foliations can be canonically determined in the geometrically reduced case.  In the next subsection we will show that this agrees with the canonically determined conductor divisor.  Thus the foliation method gives another way to obtain the conductor divisor in our set-up, and also provides a new proof of the main result of \cite{schroer_tate}. 

\begin{theorem}
\label{thm:foliation}
Let $\sX$ be a normal variety over a perfect field $k$ of characteristic $p>0$, and let $f : \sX \to T$ be a morphism to a normal variety over $k$. Let $\tau : T'\to T$ be a finite purely inseparable height one $k$-morphism from a normal variety and let $\sZ $ be the normalization of (the reduced subscheme associated to) $\sX \times_T T'$.  
Set $\sH \subseteq \sT_{T'}$ to be the foliation associated to $\tau$. 
Then the following  statements hold:
\begin{enumerate}
\item \label{itm:foliation:easy} $K_{\sZ/\sX} \sim (p-1)D$ for some Weil divisor $D$ on $\sZ$.
\item \label{itm:foliation:effective_local} There is a non-empty open set $U \subseteq T'$ and an effective divisor $C$ on $g^{-1}U$ satisfying $-C \sim D|_{g^{-1}U}$, where $g: \sZ \to T'$ is the induced morphism.

\end{enumerate}

From now assume in addition that $\sX\times_T T'$ is reduced.  Then the following hold:
\begin{enumerate}[resume]
\item\label{itm:foliation:unicity} There is natural way to choose $C$, which is unique after fixing $U$ small enough.
\item \label{itm:foliation:effective_global} If $T$ and $T'$ are regular, then  $g^* \det \sH - D \sim C$ for some effective divisor on $C$ on $\sZ$. Furthermore, if we trivialize $K_T$ on an open set, we obtain back the $C$ of point\autoref{itm:foliation:effective_local}.

\end{enumerate}

\end{theorem}

\begin{proof}

\noindent { \bfseries Proof of point \autoref{itm:foliation:easy}}:

Let $\phi$ be the morphism $\sZ\to\sX$.  By the universal properties of product and normalization, there is a morphism $\psi: \sX\to \sZ$ such that $\phi\circ\psi = \Fr$.  Therefore $\phi$ is a $k$-morphism of height one and by \cite{ekedahl_canonical} it is given by a foliation $\sF \subseteq \sT_{\sZ}$.  In particular, by \autoref{prop:foliation_canonical} we have 
$
  \omega_{\sZ/\sX}\cong (\det \sF)^{[p-1]}
$.
So taking $D$ to be any Weil divisor in the class of $\det \sF$ yields \autoref{itm:foliation:easy}.

\noindent { \bfseries Proof of point \autoref{itm:foliation:effective_local}:}

To show \autoref{itm:foliation:effective_local} we wish to show that after shrinking $T$, $H^0(\sZ, (\det \sF)^*) \neq 0$ holds. For this, it is enough to exhibit an embedding $\sF \hookrightarrow \sO_{\sZ}^{\oplus r}$ for some integer $r>0$.   Given this we get a non-zero morphism $\Det(\sF)\to \bigwedge^{\rk \sF}(\sO_{\sZ}^{\oplus r})\cong \sO_{\sZ}^{\oplus s}$ for some $s>0$, which gives a section of $H^0(\sZ,(\Det \sF)^*)$ by projection to one of the factors. In fact, when the geometric generic fiber of $f$ is reduced, then we will make a more canonical choice, which we describe in the proof of point \autoref{itm:foliation:unicity}.  However, this canonical choice does not matter for the proof of \autoref{itm:foliation:effective_local}, so we defer the explanation of it. 

We may replace $T$ by an open subset to assume that $T$ and $T'$ are both  smooth over $k$ and  $\Omega_{T'} \cong \sO_{T'}^{\oplus d}$, where $d:=\dim T' = \dim T$.
 Consider the cotangent sequence
\begin{equation*}
\xymatrix{
 g^* \Omega_{T'} \ar[r] & \Omega_{\sZ} \ar[r] & \Omega_{\sZ/T'} \ar[r] & 0.
}
\end{equation*}
Dualizing yields a diagram
\begin{equation}
\label{eq:foliation:main_diag}
\xymatrix@R=10pt{
& & \sF \ar@{^(->}[d] \\
0 \ar[r] & \sG:= \sHom_{\sZ} ( \Omega_{\sZ/T'}, \sO_{\sZ} )  \ar[r] &  \sT_{\sZ} \ar[r] & g^* \sT_{T'} \cong \sO_{\sZ}^{\oplus d}.
}
\end{equation}
We show that the composition $\sF \to \sO_{\sZ}^{\oplus d}$ is an embedding.  Equivalently we must show that $\sG\cap \sF =0$, i.e. for any open set $U$, $\sF(U)$ and $\sG(U)$ have no non-zero sections in common as subsheaves of $\sT_{\sZ}$.  Since $\sT_{\sZ}$ is torsion-free, it is enough to prove that $\sG (V) \cap \sF(V) =0$ holds on a single open set $V$ of $\sZ$.   

From now on we may restrict attention to an open subset $V\subset \sZ$.  Doing so, we may replace $T$, $T'$, $\sZ$ and $\sX$ in order to assume that they are affine and smooth (over $k$). 
We may further shrink $T$, $T'$, $\sZ$ and $\sX$ to assume that all the sheaves considered are locally free (i.e. $\sF$, $\sG$, $\sG\cap\sF$ etc), along with all the images, kernels and cokernels of the maps between them.  Finally, since $\sZ\to (\sX\times_TT')_{\red}$ is a birational morphism, we may further shrink $\sZ$ to assume that this is an isomorphism.    Under these new assumptions, we show in the remainder of the proof of point \autoref{itm:foliation:effective_local} that $\sG \cap \sF =0$. As we work with affine schemes, by abuse of notation, in what follows (in the proof of point \autoref{itm:foliation:effective_local}) all coherent sheaves denote the corresponding modules of global sections.

For a closed point $t\in T'$, denote $\sF\otimes_{\sO_{\sZ}}\sO_{{\sZ}_t}=\sF|_{{\sZ}_t}$ by $\sF_t$ and $\sG|_{\sZ_t}$ by $\sG_t$.  As all the sheaves involved are locally free, $\sF_t\to \sT_{\sZ}|_{\sZ_t}$ and $\sG_t\to \sT_{\sZ}|_{\sZ_t}$ are still embeddings, and also $(\sF\cap \sG)_t=\sF_t\cap\sG_t$, which is locally free of rank equal to that of $\sF\cap \sG$.  Therefore it is enough to show that $\sF_t\cap \sG_t=0$ for some fixed closed point $t\in T'$ over which $\sZ_t\neq \emptyset$.

Fix a closed point $t\in T'$ from now on for which $\sZ_t \neq \emptyset$.  Let $\sX':=\sX\times_T T'$.  We assumed that $\sZ\cong(\sX\times_T T')_{\red}$, so there is a surjection of structure sheaves 
$$\sO_{\sX}\otimes_{{\sO}_T}\sO_{T'}=\sO_{\sX'} \twoheadrightarrow \sO_{\sZ}.$$

Denote the structure sheaf of $t\in T'$ by $\sO_t$ (which is isomorphic to a finite extension of the perfect field $k$ together with quotient morphism $\sO_{T'}\to \sO_t$).  Let $\sZ_t$ and $\sX_t$ be the schemes $\sZ_t:=\sZ\times_{T'}t$ and $\sX_t:=\sX\times_{T} t = \sX' \times_{T'} t$.

We get the diagram
\begin{equation*}
\begin{tikzpicture}
  \matrix (m) [matrix of math nodes,row sep=3em,column sep=4em,minimum width=2em]
  {
     \sO_T & \sO_{T'} &  &  \sO_t \\
     \sO_{\sX} & \sO_{\sX'} & \sO_{\sZ} \\
      & & & \sO_{\sX_t} & \sO_{\sZ_t}\\
     };
  \path[-stealth]
    (m-1-1) edge  (m-2-1)
            edge  (m-1-2)
    (m-2-1.east|-m-2-2) edge (m-2-2)
    		edge  (m-3-4)
    (m-1-2) edge  (m-2-2)
    		edge  (m-2-3)
    		edge  (m-1-4)
    (m-1-4) edge  (m-3-5)	
    		edge  (m-3-4)	
    (m-2-2) edge  (m-2-3)
    		edge [dashed] (m-3-4)
    (m-2-3) edge  (m-3-5)
    (m-3-4) edge [dashed] (m-3-5)
            ;
\end{tikzpicture}
\end{equation*}
As $\sO_{X_t} \to \sO_{Z_t}$ can be viewed as the base-change of the surjective map $\sO_{\sX'} \to \sO_{\sZ}$, we see that it is also surjective. 
We also have
\begin{equation}
\label{eq:relative_differentials}
\sHom_{\sZ} (\Omega_{\sZ/T'}, \sO_{\sZ} ) |_{\sZ_t} \cong 
\underbrace{\sHom_{\sZ_t} \left(\Omega_{\sZ/T'} |_{\sZ_t}, \sO_{\sZ_t} \right) }_{\Omega_{\sZ/T'}  \textrm{ is locally free}}
\cong 
\underbrace{\sHom_{\sZ_t} \left(\Omega_{\sZ_t/k(t)} , \sO_{\sZ_t} \right)}_{\textrm{Base change \cite[II.8.10]{hartshorne_ag}}}.
\end{equation}
In fact, we claim more: the left and right hand sides of \autoref{eq:relative_differentials} both naturally embed in $\sT_{\sZ}|_{\sZ_t}$, and these embeddings are equal under the above isomorphism.   To see this, note that the two embeddings are equivalent to 
$$\sHom_{\sO_{\sZ_t}}(\Omega_{\sZ/T'}\otimes {\sO_{\sZ_t}},\sO_{\sZ_t})\hookrightarrow \sHom_{\sO_{\sZ_t}}(\Omega_{\sZ/k}\otimes \sO_{\sZ_t},\sO_{\sZ_t})\hookleftarrow\sHom_{\sO_{\sZ_t}}(\Omega_{\sZ_t/k},\sO_{\sZ_t})$$
and these inclusions come from dualizing the surjections
$$\Omega_{\sZ/T'}\otimes \sO_{\sZ_t}\twoheadleftarrow \Omega_{\sZ/k}\otimes\sO_{\sZ_t}\twoheadrightarrow\Omega_{\sZ_t/k}.$$
These two surjections are equal under the identification of $\Omega_{\sZ/T'}\otimes \sO_{\sZ_t}$ with $\Omega_{\sZ_t/k}$.

Furthermore, $\sT_{\sZ}|_{\sZ_t}$ can be identified with the sheaf of $k$-derivations $\sO_{\sZ} \to \sO_{\sZ_t}$, which implies together with the arguments of the previous paragraph that $\sG_t$ is exactly the sheaf of those $k$-derivations $\sO_{\sZ} \to \sO_{\sZ_t}$ that can be obtained from a $k$-derivation $\sO_{\sZ_t} \to \sO_{\sZ_t}$ by precomposing with the natural surjection $\sO_{\sZ} \to \sO_{\sZ_t}$. Let $\delta : \sO_{\sZ_t} \to \sO_{\sZ_t}$ be  a non-zero $k$-derivation of $\sO_{\sZ_t}$, corresponding then to an element of $\sG_t$ with image in  
 $\sT_{\sZ}|_{\sZ_t}$ being the composition
\begin{equation*}
\xymatrix{
\sO_{\sZ} \ar@{->>}[r] \ar@/^1.pc/[rr]^{\delta'} & \sO_{\sZ_t} \ar[r]^{\delta} & \sO_{\sZ_t} .
}
\end{equation*}
If $\sF_t \cap \sG_t \neq 0$, then by the affineness assumptions, they also have a common non-zero section. So, assume now that the above considered $\delta$ yields $0 \neq \delta' \in \sF_t$, that is, it comes from a $k$-derivation $d : \sO_{\sZ} \to \sO_{\sZ}$ in $\sF$ by composing $d$ with the natural surjection $\sO_{\sZ} \twoheadrightarrow \sO_{\sZ_t}$. By definition the derivations in $\sF$ are exactly the ones that are 0  when restricted to $\sO_{\sX}$. In particular, we have a commutative diagram as follows.
\begin{equation*}
\xymatrix{
\sO_{\sZ} \ar@{->>}[d] & \ar[l]_d \ar[ld]_{\delta'} \sO_{\sZ} \ar@{->>}[d] & \sO_{\sX} \ar@{_(->}[l] \ar@{->>}[d] \ar@/_1.5pc/[ll]_0 \ar[ld]^{\gamma} \\
\sO_{\sZ_t} & \ar[l]^{\delta} \sO_{\sZ_t} & \sO_{\sX_t} \ar@{->>}[l]
}
\end{equation*}

Let $t'$ be the image of $t$ in $T$. Since $k$ is perfect, and $k(t')$ is a finite extension, $k(t')$ is perfect too. Hence, in the factorization $k(t') \to k(t) \to k(t')$ given by the height $1$ assumption on $T$, the composition is an isomorphism, and hence so is $k(t') \to k(t)$. This then shows that $\sX \times_T t \cong \sX \times_T t'$, which then in turn implies that $\sO_{\sX} \to \sO_{\sX_t}$ is surjective. 

Since the composition $\sO_{\sX} \to \sO_{\sZ_t}$ (from the top right corner to the bottom left corner) is $0$, $\delta$ has to be also zero as $\gamma$ is surjective. This is a contradiction, which concludes our proof of \autoref{itm:foliation:effective_local}.

\noindent { \bfseries Proof of point \autoref{itm:foliation:unicity}:}

To prove \autoref{itm:foliation:unicity}, note that after restricting $T$ and $T'$ to be regular, as in the proof of point \autoref{itm:foliation:effective_local}, $\sH$ is exactly the kernel of $\sT_{T'} \to \tau^* \sT_T$. Since $\sF$ maps to $0$ in $\phi^* \sT_{\sX}$, by the commutativity of the next diagram, $\sF$ maps to $g^* \sH \subseteq g^* \sT_{T'}$ ($g^* \sH$ is a subsheaf of  $g^* \sT_{T'}$ because by restricting the base we may assume that $g$ is flat).
\begin{equation*}
\xymatrix{
\sT_{\sZ} \ar[r] \ar[d] & \phi^* \sT_{\sX} \ar[d] \\
g^* \sT_{T'} \ar[r] & g^* \tau^* \sT_T \cong \phi^* f^* \sT_T
}
\end{equation*}
After shrinking $T$ and $T'$ small enough we may also assume that $\sH \cong \sO_{T'}^{\oplus i}$, where $p^i = \deg \tau$, and also that these form the first $i$ summands of $\sT_{T'} \cong \sO_{T'}^{\oplus d}$. Furthermore, by the reducedness assumption, $\deg \tau = \deg \phi$. Since, $\deg \phi = p^{\rk \sF}$, $\rk \sF =i$ holds also.  
Hence, $\sF \hookrightarrow g^* \sH$ is generically an isomorphism. In particular, $g^* (\det \sH)^*  \to (\det \sF)^*$ is a non-zero injection, and hence a generator of $\det \sH$ yields a unique non-zero section of $( \det \sF)^*$, up to multiplication by a unit of $\sO_T$. This then defines a unique $C$, which concludes the proof of \autoref{itm:foliation:unicity}.

\noindent { \bfseries Proof of point \autoref{itm:foliation:effective_global}:}

The proof of point \autoref{itm:foliation:effective_local} and of point \autoref{itm:foliation:unicity} yield with a little more work the present point too. We may consider the diagram \autoref{eq:foliation:main_diag}  in the global (base) case too, so before restricting $T$, and its row is exact also in this case. Then, as $\sG$ and $\sF$ are torsion-free sheaves, the proof of  point \autoref{itm:foliation:effective_local} implies that $\sG \cap \sF = 0$ holds also in this global case. So, we obtain an injection $\sF \hookrightarrow g^* \sT_{T'}$, which factors through $g^* \sH$ by the proof of point \autoref{itm:foliation:unicity}. In fact, for the latter  we need to show that $g^* \sH \to g^* \sT_{T'}$ is an inclusion. Indeed, by just looking at the ranks of $\sH$, $\sT_{T'}$ and $\sT_T$, it is generically an injection, and then using that $g^* \sH$ is torsion-free we obtain that it is an injection everywhere.

In particular, by the degree argument in the proof of point \autoref{itm:foliation:unicity}, the above injection yields a non-zero homomorphism $(g^* \det \sH)^* \to (\det \sF)^*$. This corresponds to an effective divisor linearly equivalent to $g^* (\det \sH) + (\det \sF)^* = g^* (\det \sH) -D$, concluding the proof of point \autoref{itm:foliation:effective_global}.

\end{proof}

\begin{remark}
	Most likely, \autoref{thm:foliation}.\autoref{itm:foliation:effective_global} is true in greater generality, if one assumes equidimensionality of $f$. Probably then it is enough to assume that $T$ and $T'$ are normal, using reflexivization at multiple places. However, we leave the precise checking to the reader, as we do not see an important application of this set of assumptions at the present time. 
\end{remark}

\begin{proof}[Proof of \autoref{thm:generic}]
	
	We start with arbitrary field extension $k'/k$ of characteristic $p$, and make reductions to a more convenient situation. 
	
	Let $A$ be an integrally closed subring of $k$ which is a finitely generated $\bZ$-algebra, such that there is a scheme $\mathcal{X}_A$ of finite type over $\Spec(A)$ such that $X=\mathcal{X}_A\otimes_A k$. Let $\tilde{k}$ be the fraction field of $A$, so that the generic fiber of $\mathcal{X}_A\to \Spec(A)$ is $X_{\tilde{k}}$.  This satisfies $X_{\tilde{k}}\otimes_{\tilde{k}}k\isom X$.  By flat base change we have 
	$$\phi^*K_{X_{\tilde{k}}/\tilde{k}}=K_{X/k}.$$ 
	So we may replace $k$ by $\tilde{k}$ to assume that $k$ is the function field of a normal variety $\Spec(A)$ over a perfect field $l$.
	
	Next, note that the normalization of $X\otimes_k k'$ is defined over some finitely generated subextension $k''/k$.  Therefore we can replace $k'$ by $k''$ to assume that $k'/k$ is a finitely generated field extension.
	
	Now let $t_1,...,t_n$ be a transcendence basis of $k'/k$.  Then let $k^t=k(t_1,...,t_n)$, so that $k'/k^t$ is algebraic and $k^t/k$ is purely transcendental.  Let $B=A[t_1,...,t_n]$.  This is integrally closed if and only if $A$ is, because $B\to A$ is base change by $l\to l[t_1,...,t_n]$, which is smooth.  Similarly we see that $X_{k^t}$ is normal and that $K_{X_B}/B$ is the pullback of $K_{X_A}/A$.  We are therefore free to replace $A, \mathcal{X}_A, k$ etc. with $B,\mathcal{X}_B, k[t_1,...,t_n]$.  As a result, we may now assume that $k'$ is a finite algebraic extension of $k$ (as we assumed already that it was finitely generated), in addition to $k$ being the function field of a variety over a perfect field.
	
	Now let $k''$ be the perfect closure of $k$ within $k'$.  We may decompose $k''/k$ into a sequence of height one extensions: $k''=k_m/k_{m-1}/.../k_1/k_0=k$ (for example by repeatedly adjoining elements $a\in k''\setminus k$ such that $a^p\in k$).  Let $B_i$ be the integral closure of $A$ in $k_i$, and $\sX_{B_i}$ be the normalization of $\sX_{B_{i-1}}\times_{B_{i-1}}B_i$.  Then $X_{k_i}\to X_{k_{i-1}}$ is the induced map on generic fibers of the $\sX_{B_i}\to\sX_{B_{i-1}}$.  The required formula comes from repeatedly restricting the formula of \autoref{thm:foliation} to the generic fibers, and noting that the base changes and normalization commute.  Finally, we perform a base change by $k'/k''$: this is a finite separable field extension, so we have $K_{X_k'/k'}=\pi^*K_{X_k''/k''}$, and the Weil divisor $(p-1)\pi^*C$ also has coefficients divisible by $p-1$.
\end{proof}

\subsection{Equivalence of our divisor and  the conductor}
\label{subsec:algebraic_lemmas}

In this subsection, we will prove \autoref{thm:conductor}, by proving that the natural choice of $(p-1)C$ given to us in \autoref{thm:foliation}\autoref{itm:foliation:unicity} is in fact equal to the conductor of the normalization.  As a result we can deduce that this conductor will always have coefficients divisible by $p-1$, as in Tate's formula.  This will involve several preparatory commutative algebra lemmas.


\begin{notation}
	Throughout \autoref{subsec:algebraic_lemmas}, adjoining $x$ and $y$ means adjoining independent variables, and adjoining anything else means taking the ring generated by those elements inside an existing ring. 
\end{notation}

\begin{lemma}
	\label{lem:generated_subring}
	Let $(B,n)$ be a DVR with uniformizer $s$, let $D \subseteq B$ be a subring such that $n^i \subseteq D$ for some integer $i>0$, and such that the composition $D \to B \to B/n$ is surjective. Then $D[s] = B$.
\end{lemma}

\begin{proof}
	Fix $b \in B$. We want to prove that $b \in D[s]$.  It is enough to prove that \emph{for each integer $j>0 $, there is a $d_j \in D[s]$, such that $b- d_j   \in n^j$}. Indeed, then it would follow that $b - d_i \in n^i \subseteq D$, which would conclude the proof of the lemma.  
	
	So, we are left to prove the above claim. 
	For $j=0$ the statement is vacuous as we may take $d_j=0$. So, we are left to show the inductive step. Consider the image $\ov$ of $v:=\frac{b- d_{j-1}  }{s^{j-1}} \in B$ in $B/n$. As the composition $D \to B \to B/n$ is surjective, there is a $d \in D$ mapping to $\ov$. In particular, $v -d \in n$. Hence, 
	\begin{equation*}
	n^j \ni (v-d) s^{j-1} = b - d_{j-1} - d s^{j-1} .
	\end{equation*}
	Therefore, setting $d_j:=d_{j-1} + d s^{j-1}$ concludes the inductive step of our claim, and then also the claim itself. 
	
\end{proof}

\begin{lemma}  
	\label{lem:DVR_s_inseparable_structure}
	Let $(A,m)$ be a DVR over a field  $k$  of characteristic $p>0$, and $k[t]$ be a non-trivial inseparable extension of $k$ such that $t^p \in k$. Assume that $A[t]:=A \otimes_k k[t] \cong A[x]/(x^p - t^p)$ is integral, that it satisfies the property \Nagataone (its integral closure is module finite over $A[t]$), and let $B$ be the normalization of $A[t]$. Then our initial statement is that $B$ is a DVR. 
	
	Denote by $n$ the maximal ideal of $B$, and let $u$ be a generator of $n$.
	Our main statement is the following:
	\begin{enumerate}
		\item \label{itm:case_surjective} First case: if $A[t] \to B/n$ is surjective, then $B=A[t,u]$.
		\item \label{itm:case_not_surjective} Second case: if $A[t] \to B/n$ is not surjective, then:
		\begin{enumerate}
			\item \label{itm:same_generator} $u$ can be chosen to be a generator also for $m$, and hence to be in $A$,
			\item \label{itm:deg_p_field_extension} $A/m \to B/n$ is a field extension of degree $p$ and hence any element of $B/n$ not in $A/m$ generates this extension, 
			\item \label{itm:case_not_surjective:Artinian} $A[t]/m A[t]$ is an Artinian ring of length $p$, such that the kernel of $A[t]/m A[t] \to B/n$ is the  radical ideal of $A[t]/m A[t]$ and the image of $A[t]/m A[t] \to B/n$ agrees with that of $A/m \to B/n$.
		\end{enumerate}
	\end{enumerate}

\end{lemma}

\begin{proof}
	First we show our initial statement that $B$ is a DVR. Indeed, as $A\subset B$ is inseparable, finite and of height $1$, $\spec B \to \Spec A$ is a homeomorphism. Hence, $\dim B=1$ and $B$ is also local. Furthermore, as $A\subset B$ is finite, and $A$ is Noetherian, so is $B$.  So $B$ is an integrally closed local Noetherian domain of dimension $1$, which implies that it is a DVR.
	
	According to \autoref{def:conductor}, the conductor ideal exists and so there is an integer $i>0$ such that $n^i$ is also an ideal of $A[t]$. 
	
	We turn now to proving the main statements. 
	\begin{enumerate}
		\item Here we apply \autoref{lem:generated_subring} by setting $D=A[t]$, and $s=u$.
		\item We show the statements one by one:
		\begin{enumerate}
			\item

			By the fundamental theorem of finitely generated modules over a PID and the \Nagataone assumption, $B$ has to be a free module over $A$. As $\Spec B \to \Spec A[t]$ is an isomorphism at the generic point (see discussion in \autoref{subsec:conductor} about conductors of \Nagataone rings), it follows that $B$ has to have rank $p$ over $A$. In particular, $\dim_K (B/mB)=p$, where $K:=A/m$. 
			
			Now, consider the field extension $K \hookrightarrow L$, where $L= B/n$. As the composition $A[t] \to L$ is not surjective, this field extension is not trivial. However, $L$ is a quotient of $B/mB$ by its radical ideal. So, we have
			\begin{equation*}
			p = \dim_K (B/mB) = (\dim_L (B/mB)) \cdot (\dim_K L),
			\end{equation*}
			where the second factor is greater than $1$. It follows $\deg [L :K]=p$ and $B/mB \to L$ is an isomorphism. However, then $n=mB$ and hence $u$ can be chosen to be any generator of $m$ (which we assume from now).
			\item We already proved the degree $p$ part in the previous point. Then the rest follows using that $p$ is prime. 
			
			\item
			Set $R:=A[t]/mA[t]$, and let $I \subseteq R$ be the unique maximal ideal (we know there is a unique one as $\Spec A[t]$ is homeomorphic to $\Spec A$). Then $R/I$ is a field sitting as an intermediate field: $K \subseteq R/I \subseteq L$. As $[L:K]=p$, either $R/I = K$ or $R/I=L$. However, in the latter case this isomorphism would imply that $A[t] \to L$ is a surjection, which we assumed was not the case. Hence, $R/I=K$, and then $I$ is the nil-radical of $R$ (as $R$ is Artinian). In particular $I$ has to lie in the kernel of $R \to L$ (as $L$ is a field), and then $I$ is in fact equal to this kernel. This shows that the image of the map $R \to L$ is indeed $K$, which then concludes our proof. 
			
		\end{enumerate}

	\end{enumerate}
	
\end{proof}

\begin{remark}
	Both cases of \autoref{lem:DVR_s_inseparable_structure} occur in examples:
	\begin{enumerate}
		\item The simplest examples are generally of this form.  For example, consider the localization at the maximal ideal $(x,y^2+s)$ of $A=k[x,y]/(y^2+x^3+s)$ where $k$ is a field of characteristic $2$ and $s \in k \setminus k^2$, and take the extension to be $k\left( \sqrt{s} \right)$.
		
		\item An example of this type is given by $A=k[x,y]/(sx^2+xy^2+1)$, where $k$ is a field of characteristic $2$ and $s \in k \setminus k^2$.  The non-smooth locus is given by the vanishing of $y$.  $m:=(y)$ is a maximal ideal of $A$ because $A/m\cong k[s^{1/2}]$, and so $A_m$ is a regular DVR.  Replace $A$ by $A_m$.
		
		Letting $t=s^{1/2}$, $k'=k[t]$ and $x'=tx+1$, we have $$A[t]\cong (k'[x',y]/(x'^2+(t^{-1}x'+t^{-1})y^2))_{(y,x')}$$
		and 
		$$A[t]/yA[t]\cong k'[x']/(x'^2).$$  
		If we let $u=\frac{x'}{y}$ then $$B\cong (k'[u,y]/(u^2+t^{-1}uy+t^{-1}))_{(y)}$$
		where $y$ is still a uniformizer, and $B/n=B/yB\cong k'[s^{1/4}]$.

	\end{enumerate}
\end{remark}

\begin{lemma}
	\label{lem:derivation_divided_by_u}
	Let $(B,n)$ be a DVR with uniformizer $u$.  Let $A$ be a subring of $B$ and suppose we can write $B$ as  $B=A[r_1,\dots,r_d]$ for finitely many elements $ r_i\in B$. Assume that there is a $\Delta \in \Der_A(B)$ such that $\Delta(r_i)=u^{a_i} r_i$ for some integers $a_i>0$. Then, $\Delta/u \in \Der_A(B)$.
\end{lemma}

\begin{proof}
	We have to show that for every $b \in B$, $\Delta(b) \in n$. However, we may write $b = q(r_1,\dots,r_d)$ for some $q(x_1,\dots,x_d) \in A[x_1,\dots,x_n]$. Then our assumptions show that $\Delta(b) = \Delta(q(r_1,\dots,r_d)) \in n$.
\end{proof}

\begin{lemma}
	\label{lem:conductor}
	In the situation of \autoref{lem:DVR_s_inseparable_structure} assume also that $B \neq A[t]$, and we are given a $\Delta \in \Der_A(B)$, such that
	\begin{enumerate}
		\item $\Delta/u$, as a function $B\to K(B)$, does not define an element of $\Der(B)$, 
		and
		\item $\Delta(t) = u^a$ for some integer $a>0$.
	\end{enumerate}
	Let $c>0$ be the smallest integer such that $n^c \subseteq A[t]$. Then $(p-1)a \leq c$.
\end{lemma}

\begin{proof}
	We present separate proofs for the two cases of \autoref{lem:DVR_s_inseparable_structure}. 
	
	\noindent {\bfseries Case \autoref{itm:case_surjective} of \autoref{lem:DVR_s_inseparable_structure}}:
	
	First note that $\Delta(u)$ is a unit in $B$. Indeed, if it was not, then using that $B=A[t,u]$ (\autoref{lem:DVR_s_inseparable_structure}), \autoref{lem:derivation_divided_by_u} would give a contradiction.
	
	Let $\Theta$ be the following set of integers:
	$$\Theta:=\{j\in\mathbb{N}_{>0} \  | \  p \nmid j, \exists r\in B^{\times} \mathrm{\ s.t.\ } ru^j\in A[t]\}$$
	$\Theta$ is not empty because the conductor is non-trivial, that is we know that $n^c=Bu^c\subseteq A[t]$.  Furthermore, we see that $d\in\Theta$ for $d\geq c$ whenever $p\nmid d$.  Let $j_0$ be the minimal element of $\Theta$.  
	We claim that:
	\begin{enumerate}[label=(\roman*)]
		\item\label{itm:conductor:decrease_n}
		If $j\in \Theta$ then $j - (a+1) \geq 0$, and  $j-(a+1)\in \Theta$ unless $p|(j-(a+1))$.
		\item\label{itm:conductor:bound}
		$a+1\leq j_0$.
		\item\label{itm:conductor:divisibility}
		$a+1\equiv j_0\mathrm{\ (mod\ }p)$.  In particular, $p\nmid a+1$.
	\end{enumerate}
	Note that \ref{itm:conductor:bound} and \ref{itm:conductor:divisibility} are consequences of \ref{itm:conductor:decrease_n}, so we only need to prove \ref{itm:conductor:decrease_n}. Suppose that $j\in\Theta$ and so we can let $r\in B^{\times}$ be such that $ru^j\in A[t]$.  Let $q(x)\in A[x]$ be such that $ru^j=q(t)$.  Apply $\Delta$ to the equation $ru^j-q(t)=0$, to obtain:
	\begin{equation}
	\label{eq:conductor:derivative_computation}
	0=\Delta(ru^j-q(t))=
	\underbrace{u^j\Delta(r)+rju^{j-1} \Delta(u)-u^aq'(t)}_{\Delta(A)=0, \textrm{ and } \Delta(t) = u^a}
	= u^{j-1}(u \Delta(r) + rj \Delta(u)) - u^a q'(t). 
	\end{equation}
	As all three of $r$, $j$ and $\Delta(u)$ are units, so is $u \Delta(r) + rj \Delta(u)$. Hence, $u^{j-1} \in B u^a$, and therefore $j-1 -a \geq 0$. Applying \autoref{eq:conductor:derivative_computation} again shows that $j-1-a \in \Theta$, unless $p |(j-1-a)$.
	
	\emph{Suppose now that $c<(p-1)a$.}  Then, as $(p-1)a-1\equiv -(a+1)\not\equiv 0\mathrm{\ {mod}\ } p$ according to \ref{itm:conductor:divisibility}, we have $(p-1)a-1\in\Theta$.  We will use property \ref{itm:conductor:decrease_n} to derive a contradiction to property \ref{itm:conductor:bound}. 
	
	First we list some elementary properties of the set of integers $N_m:=(p-1)a-1-m(a+1)$:
	\begin{enumerate}[resume*]
		\item\label{itm:conductor:in_theta}
		$p \nmid N_m$  for any $m\in\{0,...,p-2\}$, and 
		\item\label{itm:conductor:in_range}
		$N_m\in\{1,...,a\}$ for some $m\in \{0,...,p-2\}$.
	\end{enumerate}
	Property \ref{itm:conductor:in_theta} follows immediately from $$N_m\equiv -(m+1)(a+1) \mathrm{\ (mod\ }p),$$ for $p$ does not divide $m+1$ or $a+1$, according to \ref{itm:conductor:divisibility}.
	To see property \ref{itm:conductor:in_range} note first that for any $m$ in this set, either $N_m<0$ or $N_m>0$ because $p\nmid N_m$.  Thus for the particular value $m'$ for which $N_{m'+1}<0$ and $N_{{m'}}=N_{m'+1}+(a+1)>0$ we must have $N_{m'}\leq a$ (there is such an $m'$ with $m' \leq p-2$ because $N_{p-1}=-p$).
	
	To derive the required contradiction, note that property \ref{itm:conductor:in_theta} implies that $N_m\in\Theta$ for all values of $m\in\{0,...,p-2\}$ such that $N_m>0$ by property \ref{itm:conductor:decrease_n}, while property \ref{itm:conductor:in_range} produces a contradiction to property \ref{itm:conductor:bound}.
	
	\noindent { \bfseries Case \autoref{itm:case_not_surjective} of \autoref{lem:DVR_s_inseparable_structure}}: 
	
	As in Case \autoref{itm:case_surjective}, $B$ is a DVR with maximal ideal $(u)$ and the conductor is equal to $B u^c$, where $c$ is the smallest number such that $B u^c \subseteq A[t]$. In Case \autoref{itm:case_not_surjective}, $c$ is also the smallest integer $n$ such that $B^{\times} u^c \subseteq A[t]$, because $u\in A$ by \autoref{lem:DVR_s_inseparable_structure} \autoref{itm:same_generator}.  Fix $r \in B^\times$, and assume that $d_r$ is the smallest integer such that $u^{d_r} r \in A[t]$.  For any $r$, this satisfies $d_r\leq c$.  We may write $u^{d_r} r = \sum_{i=0}^{p-1} a_i t^i$ where $a_i \in A$ and are unique (as $A[t]$ is a free module over $A$ with generators $1,\dots, t^{p-1}$).
	
	Now apply the derivation $\Delta$, $p-1$-times, to get:
	
	\begin{equation*}
	u^{d_r}  \Delta^{p-1}(r) = a_{p-1} (p-1)! u^{a (p-1)}
	\end{equation*}
	
	If $\Delta^{p-1}(r)$ is a unit, we can deduce that $a(p-1)\leq d_r\leq c$, which would finish our proof.  Thus we need only prove the following claim:
	
	\noindent \emph{Claim:}
	There is some element $r \in B^\times$ such that $\Delta^{p-1}(r )$ is a unit.
	
	Let $L= B/n$, $K:=A/m$, and $\phi : B \to L$ be the natural quotient map.  Then we have the diagram:
	\begin{equation*}
	\xymatrix@R=3pt{
		\Der_A (B) \ar[r] &  \Der_A( B, L)   & \ar[l]^{\cong} \Der_K(L) \\
		\Delta \ar@{|->}[r] & \phi \circ \Delta   &  \\
		&  d \circ \phi &  \ar@{_(->}[l] d
	}
	\end{equation*}
	It is immediate from the definition that the right hand map is injective.  To show it is surjective, take some $D\in \Der_A(B,L)$.   $D(u)=0$, because $u\in A$, and so the Leibniz rule implies that $D(uB) = 0$.  Thus any such $D$ does descend to a well defined derivation $d\in \Der(L)$.  $d$ is a $K$ derivation because $K$ is the image of $A$ under $\phi$.
	
	As we assumed that $\Delta/u\notin\Der(B)$, it follows that the element $d \in \Der_K(L)$  with $d\circ\phi=\phi\circ\Delta$ is not zero. 
	Thus to prove the claim, it is enough to show that for any non-zero element $d\in\Der_K(L)$, there is some $s \in L$ such that $d^{p-1}(s) \neq 0$.
	
	We have $L = K[z]$ for some $z \in L$, which exhibits $L \cong K[x]/(x^p - v)$, where $0 \neq v := z^p \in K$.  One non-zero element of  $\Der_K(L)$ is $\frac{\partial}{\partial x}$.  Because the extension is of degree $p$, $\Der_K(L)\cong L$, and so $d = \lambda \frac{\partial }{\partial x}$ for some $\lambda \in L^{\times}$. So, it is enough to find $s \in L$, such that $\left(\frac{\partial}{\partial x}\right)^{p-1}(s) \neq 0$.  But  $x^{p-1}$ is such an element. 
	
\end{proof}

\begin{lemma}
	\label{lem:multiplication_inclusion}
	Let $A$ be an excellent integral ring and $K \to A$ a ring homomorphism from a field. Let $K \hookrightarrow K_1 \hookrightarrow K_2$ be finite field extensions, such that $A \otimes_K K_1$ and $A \otimes_{K} K_2$ are integral.  Set $B$ and $C$ to be the normalizations of $A \otimes_K K_1$ and $A \otimes_K K_2$, respectively. Further introduce the following notation:
	\begin{itemize}
		\item $\sC(C/A)$ is the conductor of $A \otimes_K K_2 \hookrightarrow C$,
		\item $\sC(C/B)$ is the conductor of $B \otimes_{K_1} K_2 \hookrightarrow C$, and
		\item $\sC(B/A)$ is the conductor of $A \otimes_K K_1 \hookrightarrow B$.
	\end{itemize}
	Then we have:
	\begin{equation}
	\label{eq:multiplication_inclusion}
	\sC(C/A)\supset\sC(C/B)\cdot\sC(B/A).
	\end{equation}
	as ideals of $C$.
\end{lemma}

\begin{proof}
	Note that the assumptions that $A \otimes_K K_1$ and $A \otimes_{K} K_2$ are integral are equivalent to requiring that $\Frac(A) \otimes_K K_1$ and $\Frac(A) \otimes_{K} K_2$ are the fraction fields of the former two rings. Hence, $B$ is contained in $\Frac(A) \otimes_K K_1$, and then $B \otimes_{K_1} K_2$ in $\Frac(A) \otimes_K K_2$. This shows that $B \otimes_{K_1} K_2$ is integral, and therefore $C$ is also the normalization of the latter ring. 
	
	So, the rings considered here are contained in each other as shown on the following diagram.
	\[
	\begin{tikzcd}
	A \arrow[r,hookrightarrow] & A\otimes_K K_1 \arrow[d,hookrightarrow] \arrow[r,hookrightarrow] & A\otimes_K K_2=(A\otimes_KK_1)\otimes_{K_1}K_2\arrow[d,hookrightarrow] & \\
	& B\arrow[r,hookrightarrow] & B\otimes_{K_1}K_2 \arrow[r,hookrightarrow] & C\\
	\end{tikzcd}
	\]
	Note also that as although $\sC(B/A)$ is not an ideal of $C$ (but of $B$), as $\sC(C/B)$ is an ideal of $C$ so is $\sC(C/B)\cdot\sC(B/A)$.
	Recall that by definition: 
	$$\sC(C/A):=\{a'\in A\otimes_K K_2:a' c\in A\otimes_K K_2 \textrm{ for }\forall c\in C\}$$
	Then note that  $\sC(C/B)\cdot\sC(B/A)$ is actually in $A\otimes_K K_2$.  To see this, let $$b'=\sum b_i\otimes \lambda_i\in\sC(C/B)\subset B\otimes_{K_1}K_2$$ and let $a'\in \sC(B/A)\subset A\otimes_{K}K_1$. By definition of $\sC(B/A)$, we see that $$a'b'=a'\left(\sum b_i\otimes\lambda_i\right)=\sum((a'b_i)\otimes\lambda_i)\in (A\otimes_K K_1)\otimes_{K_1} K_2$$
	as required. Furthermore, if $c\in C$, then $b'c\in B\otimes_{K_1}K_2$, and we just saw that this implies that $a'b'c\in (A\otimes_K K_1)\otimes_{K_1} K_2=A\otimes_K K_2$.  This completes our proof.
	
\end{proof}

\begin{corollary}
	\label{cor:base_change_conductor}
	Let $f : \sX \to T$ be a proper surjective morphism which is either 
	\begin{enumerate}
		\item  a map from a normal variety to another variety over $k$, or
		\item the structure morphism of a normal variety over $T= \Spec K$, where $K$ is a field of positive characteristic. 
	\end{enumerate}
	Let $T' \to S \to T$ be finite base changes such that $\sX \times_T S$ and $\sX \times_T T'$ are integral. Let $\sY$ and $\sZ$ be the normalizations of the latter varieties, with $\tau : \sZ \to \sY$ being the maps between them. Denote by $D_{\sZ/\sX}$, $D_{\sY/\sX}$ and $D_{\sZ/\sY}$ the corresponding conductors. Then, we have $D_{\sZ/\sX} = D_{\sZ/\sY} + \tau^* D_{\sY/\sX}$.
\end{corollary}

\begin{proof}
	As $f$ is proper, it is enough to prove that $D_{\sZ/\sX} \leq D_{\sZ/\sY} + \tau^* D_{\sY/\sX}$, and this is exactly the statement of \autoref{lem:multiplication_inclusion} applied at points of codimension one.
\end{proof}

\begin{theorem}\label{thm:p-1C_conductor}
	
	Using the assumptions and notation of \autoref{thm:foliation}, and if in addition $f$ is proper and $\sX\times_T T'$ is reduced, the natural choice of $(p-1)C$ constructed in \autoref{itm:foliation:unicity} of \autoref{thm:foliation} is equal to the divisorial part of the conductor of $\sZ\to \sX\times_T T'$ over a restricted open set of $T$.
	
\end{theorem}

\begin{proof}

First, \emph{assume that the morphism $T'\to T$ ~(and therefore also $\sZ\to\sX$) is of degree $p$}, and thus $\mathcal{H}$ (and also $\sF$) has rank $1$. Let $E$ be the divisorial part of the  conductor of $\sZ \to \sX \times_T T'$. We shrink $T$ so that both $T$ and $T'$ are regular, $E$ and $C$ have only horizontal components, $f$ is flat and the relative Gorenstein locus of $\sX \to T$ is relatively big (that is, it is big in every fiber). The latter is possible as $\sX$ is normal and hence it has a regular big open set. Using \autoref{lem:two_conductors} along with base-change of $K_{\sX/T}$ for flat, Gorenstein morphisms \cite[Thm 3.6.1]{conrad} and also point \autoref{itm:foliation:effective_local} of the present proof,   we have $E \sim (p-1)C \sim -K_{\sZ/\sX}$. So, as $f$ is proper,  it is enough to prove that $E \geq (p-1)C$. The latter  question is local at the codimension $1$ points.

So, fix a codimension $1$ point $z \in \Supp C$. Then $z$ maps to $\eta$ (the generic point of $T'$).  Let $x \in \sX$ be the image of $z$. Set then $A:= \sO_{\sX,x}$, $B:= \sO_{\sZ,z}$ and $K:=K(T)$. Let $u \in B$ be a generator of the maximal ideal. Since $K(T')$ is a degree $p$ inseparable extension  of $K$, $K(T')= K(t)$ for any $t \in K(T') \setminus K$, and furthermore for all such $t$ we have $t^p \in K$. In particular, there is a unique $\Delta' \in \Der_K K(T')$, such that $\Delta'(t)=1$. As $\Delta' \neq 0$, and  $\dim_{K(T')} \Der_K K(T')=1$, $\Delta'$ generates $\Der_K K(T') \cong\sH_{\eta}$. In particular,  $\Delta'$ also generates $(g^* \sH)_z$. Fix also a generator $\Delta \in \sF_z \subseteq \sT_{\sZ,z} \cong \Der_k(B) = \Der (B)$ (where the latter equality is given by the perfectness of $k$). Then $\Delta(t) = wu^a$ for some integer $a \geq 0$  and $w \in B^\times$. Replacing $\Delta$ by $\Delta/w$ we may assume that $\Delta(t) = u^a$. Now, one can compute the multiplicity of $C$ at $z$ by understanding the equation of the map $\xi : \sF_z \to (g^* \sH)_z$. That is, as both are rank $1$ free $B$-modules, $\xi$ is given (up to multiplication by an element of $B^\times$) by a multiplication by $u^b$, where $b \geq 0$ is an integer.
The multiplicity of $C$ at $z$ is then $b$. \emph{We claim that $a =b$.} Indeed, $\xi(\Delta) = v u^b \Delta'$ for some $v \in B^\times$, as $\Delta$ and $\Delta'$ are generators (as $B$-modules) of $\sF_z$ and $(g^* \sH)_z$, respectively. The following computation shows our claim
\begin{equation*}
u^a = \Delta(t) = \xi(\Delta)(t) = v u^b \Delta'(t) = v u^b .
\end{equation*}
After these initial considerations, applying \autoref{lem:conductor} to the above situation yields exactly that $E \geq (p-1)C$ at $z$.

\emph{Now we prove the general case, i.e. where $\deg \tau > p$}.  In that case, we can decompose $\tau$ into degree $p$ morphisms by repeatedly adjoining elements of $K(T')$ whose $p^{\mathrm{th}}$ power is in $K(T)$, and then taking the normalization of $T$ in the new field.  We must show that both the conductor and foliation divisors obtained by composing two base changes are equal to those obtained by doing it all at once. We show this first for the foliation divisor. We note that since the whole composition has the property that the pullback of the total space is reduced, the same holds for all the intermediate steps. 

Suppose we have a decomposition of a height $1$ $k$-morphism: $$T'\xrightarrow{\phi} S\xrightarrow{\psi} T$$ where $\phi$ has degree $p$.  Let $\sH$ be the foliation on $T'$ corresponding to $\psi\circ\phi$, $\sH_\psi$ the foliation on $S$ corresponding to $\psi$ and $\sH_\phi$ that on $T'$ corresponding to $\phi$.
Let $\sY$ be the normalization of $\sX\times_T S$ and $\sZ$ the normalization  of $\sY\times_S T'$.  By the universal properties, the composition $\sZ\to\sY\to\sX$ is the normalization of $\sX\times_TT'$. Use the names $\alpha: \sZ \to \sY$ and $\beta: \sY \to \sX$ for the induced morphisms. Let $\sF$ be the foliation corresponding to the morphism $\sZ\to \sX$, $\sF_\phi$ be the foliation corresponding to $\sZ\to \sY$ and $\sF_\psi$ that corresponding to $\sY\to \sX$.  By  shrinking the bases small enough and then discarding a closed set of codimension at least two from the total spaces, we may assume the following:
\begin{enumerate}
\item All the vertical morphisms are flat.
\item $\sH_\phi \subseteq \sT_{T'}$, $\sH \subseteq \sT_{T'}$, and $\sH_\beta \subseteq \sT_T$ are subbundles, and then it follows that after pulling back the above subsheaves stay subsheaves.
\item  $\alpha^* \ker \left( \sT_{\sY} \to \beta^* \sT_{\sX} \right) = \ker \left( \alpha^* \sT_{\sY} \to \alpha^*\beta^* \sT_{\sX}\right)$. To obtain the latter it is enough to guarantee that the image of $ \sT_{\sY} \to \beta^* \sT_{\sX}$ is a subbundle. This holds over the locus where both $\sX$ and $\sY$ are regular  by applying \cite[last paragraph of Section 38.a]{matsumura} and \cite[Theorem in Section 3]{kimura_niitsuma} to both $\beta$ and also to the map $\xi : \sX \to \sY$ given by the height $1$ assumption. Indeed, assuming that $\sX$ and $\sY$ are regular,  the above references give that both $\Omega_{\sY/\sX}$ and $\beta^* \Omega_{\sX/\sY}$ are locally free. However, these sheaves fit into a commutative diagram with exact row as follows:
\begin{equation*}
\xymatrix{
& \beta^* \Omega_{\sX/\sY} \ar@{^(->}[d] \\
\beta^* \Omega_{\sX} \ar@{->>}[ur]  \ar[r] & \Omega_{\sY}  \ar[r] & \Omega_{\sY/ \sX} \ar[r] & 0
}
\end{equation*}
Using that all the sheaves in the diagram are locally-free, it follows that the image of $\sT_{\sY} \to \beta^* \sT_{\sX}$ is $\beta^* \sT_{\sX/\sY}=\beta^* \sHom (\Omega_{\sX/\sY},\sO_{\sX})$, which is a subbundle, if so is $\ker (\Omega_{\sX} \to \Omega_{\sX/\sY})$. Now, looking at the other similar diagram
\begin{equation*}
\xymatrix{
& \xi^* \Omega_{\sY/\sX} \ar@{^(->}[d] \\
\xi^* \Omega_{\sY} \ar@{->>}[ur]  \ar[r] & \Omega_{\sX}  \ar[r] & \Omega_{\sX/ \sY} \ar[r] & 0
}
\end{equation*}
shows that this kernel is $\xi^* \Omega_{\sY/\sX}$, which is indeed a subbundle. 
\item $\phi^* \ker \left( \sT_{S} \to \psi^* \sT_{T} \right) = \ker \left( \phi^* \sT_{S} \to \phi^* \psi^* \sT_{T}\right)$. For this, as in the previous point, it is enough to guarantee that the image of $ \sT_{\sY} \to \psi^* \sT_{\sX}$ is a subbundle. However, here we may shrink the bases arbitrarily, so we may assume it without any worries (the issue in the previous point was that we were allowed to discard only a codimension at least $2$ closed set). 
\end{enumerate}
We have a commutative diagram as follows.
\begin{equation*}
\xymatrix{
\sT_{\sZ} \ar[r] \ar[d] & \alpha^*\sT_{\sY}\ar[r]\ar[d] & \alpha^*\beta^* \sT_{\sX} \ar[d] \\
g^* \sT_{T'} \ar[r] & g^*\phi^*\sT_{S}\ar[r]  & g^* \phi^*\psi^* \sT_T \cong \alpha^*\beta^* f^* \sT_T
},
\end{equation*}
For each row, the kernels of the two arrows and the kernel of the composition fit into a short exact sequence. Using then the commutativity of taking kernel and applying pullback shown above we obtain
another commutative diagram with exact rows as follows. 
\begin{equation*}
\xymatrix{
0\ar[r] & \sF_\phi\ar[d]\ar[r] & \sF \ar[r] \ar[d] & \alpha^*\sF_\psi\ar[d] \ar[r] &0  \\
0\ar[r] & g^*\sH_\phi\ar[r] & g^* \sH \ar[r] & g^*\phi^*\sH_\psi  \ar[r] & 0
}.
\end{equation*}

Further restricting the bases we may assume that all vector bundles in the bottom row are trivial as well as the extension. So, after also dualizing, we obtain:
\begin{equation*}
\xymatrix{
0\ar[r] & (\alpha^*\sF_\psi)^*\cong\alpha^*(\sF_\psi^*) \ar[r] & \sF^* \ar[r] &  \sF_\phi^* \ar[r] &0  \\
0\ar[r] & \sO_{T'}^{\log_p \deg \psi} \ar[r] \ar[u] & \sO_{T'}^{1+\log_p \deg \psi}  \ar[u] \ar[r] &  \sO_{T'} \ar[u] \ar[r] & 0
}.
\end{equation*}
Taking determinants yields then the following commutative diagram.
\begin{equation*}
\xymatrix{
\det\sF^* \ar@{}[r]|-\cong & \det\sF_\phi^*\otimes\alpha^*\det\sF_\psi^*. \\
\sO_{T'} \ar[u] \ar@{}[r]|-\cong & \sO_{T'} \otimes \sO_{T'} \ar[u] \\
}
\end{equation*}
where the left arrow is the map $\iota$ giving the $C$ associated to $T' \to T$, and the right one is $\iota_\phi \otimes \alpha^* \iota_\psi$, where $\iota_\phi$ and $\iota_\psi$ give $C_\phi$ and $C_\psi$ associated to $\phi$ and $\psi$, respectively. It follows then that $C = C_\phi + \alpha^* C_\psi$. 
  
The corresponding statement for the conductor is shown in \autoref{cor:base_change_conductor}.

\end{proof}

\begin{remark}
We do not know if the properness assumption in \autoref{thm:p-1C_conductor} is essential or not. As our goal is to apply to proper fibrations, we did not investigate the non-proper case. 
\end{remark}

\begin{proof}[Proof of \autoref{thm:conductor}]
The additional statement about equality with the conductor in the geometrically integral case follows from  \autoref{thm:p-1C_conductor}, together with the compatibility (of the divisorial part) of the conductor with passage to generic fiber and with composition of base-changes. The former compatibility follows from \autoref{lem:conductor} and the latter from \autoref{cor:base_change_conductor}.  Note that the final separable extension produced zero contribution to either the conductor or the foliation divisor, and \autoref{cor:base_change_conductor} implies that the pullback of the conductor through this base change remains the conductor.  
\end{proof}

\begin{proof}[Proof of \autoref{cor:singularities_of_general_fiber}]

Denote the $n^{\mathrm{th}}$ power of the relative Frobenius$/k$ by $F^n:T^{(n)}\to T$.  If we choose $n\gg 0$, then the normalization $\sY$ of $\sX\times_T T^{(n)}$ has geometrically normal generic fiber.  Thus its general fiber is normal by \autoref{prop:fibers}.  The general fibers of $\sX'=\sX\times_T T^{(n)}\to T^{(n)}$ are isomorphic to those of $\sX\to T$ as we work over a perfect field.   Let $t$ be a general point of $T^{(n)}$ and let $\sX'_{t}$ and $\sY_t$ be the fibers of $\sX'/T'$ and $\sY/T'$ respectively.  Let $\pi:\sY\to\sX'$ be the normalization.

 We first show that the coefficient of every component of the conductor of $\sY\to \sX'$ which dominates $T$ is divisible by $p-1$.  Let $\sX_0=\sX$, and define $\sX_i$ to be the normalization of $\sX_{i-1}\times_{T^{(i-1)}} T^{(i)}$.  It follows from the universal properties of product and normalization that $\sX_n=\sY$.
Now by \autoref{thm:foliation} applied to $T^{(i)}\to T^{(i-1)}$ for each $i$, we see that the horizontal coefficients the conductor are divisible by $p-1$.

Next we claim that the normalization of a general fiber of $\sX\to T$ is isomorphic to a general fiber of $\sY\to T^{(n)}$.   First note that a general fiber of $\sX/T$ is isomorphic to the corresponding fiber of $\sX'/T^{(n)}$ because we work over a perfect field.  The property of geometric normality is open on the base \cite[12.2.4]{EGA_IV_III}, so a general fiber of $\sY\to T^{(n)}$ is (geometrically) normal.  This makes it the normalization of the corresponding fiber of $\sX\to T$ as it is a normal variety which comes with a finite birational map to the fiber of $\sX'\to T^{(n)}$.

Finally, we claim that the conductor of the normalization of the general fiber is equal to the restriction of the conductor of $\sY\to \sX'$.  
As $\sX$ is normal we may shrink $\sX$ by removing a subvariety of codimension $2$ to ensure that $\sX$ is regular and so $\omega_{\sX}$ is a line bundle.   $T^{(n)}\to T$ is flat and so $\omega_{\sX'/T^{(n)}}\cong\phi^*\omega_{\sX/T}$.   We may further remove a codimension $2$ subvariety of $\sX$ and $\sY$ in order to assume that $\sY$ is regular.  The result of these restrictions is that we have removed a codimension at most $2$ part of the general fibers, which we are allowed to do.  

If we now restrict to any affine subset of our set-up, we saw in \autoref{lem:two_conductors} that the conductor ideal is equal to the image of the trace map $\zeta:\mathrm{Hom}_R(S,R)\to S$.  The trace map is compatible with restriction to general fibers by \cite[Lemma 2.18]{psz}, and so we deduce that the global conductor restricts to the conductor on the general fibers.

\end{proof}

\section{Geometric consequences}
\label{sec:geom_conseq}

We first apply the theorem to our statement about singularities in fibrations of arbitrary dimensions.

\begin{proof}[Proof of \autoref{cor:general_dimension}]

First recall from \autoref{prop:fibers} that a general fiber of a morphism is normal (resp. smooth) if and only if the geometric generic fiber is normal (resp. smooth).  Therefore by \autoref{prop:fibers} and \autoref{lem:fibre_normalisation} we may replace ``general'' with ``geometric generic'' in the statement of \autoref{cor:general_dimension}.

First let us deal with the case $K_X \equiv 0$.  Let $X$ be the generic fiber of $f$, defined over the field $K=K(T)$. $X$ is normal because $\sX$ is normal.  Let $\phi:Y\to X$ be the normalization of the reduced subscheme of the base change by $\overline{K}/K$, and suppose that $Y$ is smooth in order to obtain a bound on $p$.   By  \autoref{thm:generic}, $K_Y+(p-1)C\equiv 0$ for some non-zero effective Weil divisor $C$.  We know that $C$ is non-zero by \cite[4.2]{tanaka_behaviour}, because $K(T)$ is algebraically closed in $K(X)$.  By the smoothness of $Y$, $C$ is Cartier.  Fix a general curve $\Gamma$ which has positive intersection with $C$.  By applying bend-and-break in the form of \cite[1.13]{kollar_mori} there is a rational curve $\Gamma'$ which satisfies 
$$0<-K_{Y}\cdot\Gamma'=(p-1)C\cdot\Gamma'\leq\dim(X)+1.$$
But since $C$ is Cartier, $C\cdot \Gamma'$ is an integer and so we obtain the bound $p\leq\dim(X)+2$.

Now suppose that $\rho(\sX/T)=1$ and $-K_X$ is ample over $T$.   Let $\phi:\sY\to \sX$ be the normalized base change of $f:\sX\to T$ by Frobenius.  By \autoref{lem:base_change_rho}, $\rho(Y/T)=1$.  By \autoref{thm:foliation} we have the formula
$$K_{\sY}+(p-1)\sC\sim_{\bQ} \phi^*K_{\sX}$$
where $\sC$ is a $\bQ$-Cartier divisor which is ample over $T$, because it is effective and $\rho(Y/T)=1$.
Now restricting to the generic fibres $Y$ and $X$ over $K=K(T)$, we find that 
 $$K_Y+(p-1)C=\phi^*K_X$$ for some ample Cartier divisor $C$.  

 Now, let us pass to $\overline{K}$ to be able to apply bend and break with bounds on anti-canonical degrees. Denote the corresponding maps, divisors and varieties by tilde. That is, $\widetilde{Y}:=Y_{\overline{K}}$, $\widetilde{C}:=C_{\overline{K}}$ and $\widetilde{\phi} : \widetilde{Y} \to X$ is the induced morphism.  Note that $\widetilde{C}$ is also an ample $\mathbb{Z}$-divisor. Then 
we can choose some curve $\Gamma$ on some component of $\widetilde{Y}$ with 
\begin{equation*}
K_{\widetilde{Y}} \cdot \Gamma = \left((p-1)\widetilde{C}-\widetilde{\phi}^*K_X\right)\cdot\Gamma \geq (p-1)\widetilde{C}\cdot\Gamma>0, 
\end{equation*}
 and hence by bend and break obtain $\Gamma'$ with $$-(\dim(X)+1)\leq K_{\widetilde{Y}}\cdot\Gamma'=\left(\widetilde{\phi}^*K_X-(p-1)\widetilde{C} \right)\cdot\Gamma'< -(p-1)\widetilde{C}\cdot\Gamma'<0.$$  From this we obtain a bound on $p$ as before, but this time $p\leq \dim(X)+1$ because $\widetilde{\phi}^*K_X\cdot\Gamma'<0$.

\end{proof}

Next we prove the statements about del Pezzo surfaces:

\begin{theorem}[\autoref{cor:normality_short}]\label{cor:normality}
	Let $X$ be a normal, projective and Gorenstein surface over a field $k$ of characteristic $p>0$, with $k$ algebraically closed in $K(X)$, and such that $-K_{X}$ is ample.  If $p>3$ then $X$ is geometrically normal.
	
	Furthermore, if $Y$ is the normalization of  the reduced subscheme of $X_{\overline{k}}$, and $C$ is an integral divisor such that $(p-1)C\lin -K_{Y/X}$, then $(p,Y,C)$ is one of the following:
	
	\begin{itemize}
		\item $(3,\PP^2,L)$
		\item $(3,S_d, F)$ for $d\geq 2$
		
		\item $(2,\PP^2,L)$
		\item $(2,\PP^2,C\in|2L|)$
		\item $(2,\mathbb{P}^1\times\mathbb{P}^1, C\in |F_1+F_2|)$
		\item $(2,\mathbb{P}^1\times\mathbb{P}^1, F_i)$
		\item $(2,H_d,C\in |D+F|)$ for $d\geq 1$
		\item $(2,H_d,D)$ for $d\geq 1$
		\item $(2,S_2,C)$, where $C_{H_2}\in |D+2F|$ where $C_{H_2}$ is the birational transform of $C$
		\item $(2,S_d, 2F)$ for $d\geq 2$
	\end{itemize}
	
	Here we denote the Hirzebruch surface of degree $d$ by $H_d$, its fiber and exceptional section by $F$ and $D$, and the surface obtained by contracting $D$ by $S_d$.
	
\end{theorem}

\begin{proof}[Proof of \autoref{cor:normality}]
Let $X$ be as in \autoref{cor:normality}, except that we enlarge $k$ to assume that it is algebraically closed in $K(X)$.  Assume $X$ is not geometrically normal.  Let $Y$ be the normalization of an irreducible component of the reduced subscheme of $X\otimes_k\overline{k}$.  Let $\pi:Y\to X$ be the  natural morphism. \autoref{thm:generic} implies that there is an 
 effective Weil divisor $C$ such that $K_Y+(p-1)C\lin \pi^*K_X$ (as Weil divisors).  As $X$ is not geometrically normal, by factoring $\overline{k}/k$ via $k^{\frac{1}{p^\infty}}/k$ and applying \cite[4.2]{tanaka_behaviour} we obtain that $C\neq 0$. Let $g:Y'\to Y$ be the 
minimal resolution of $Y$, and let $g^*(K_Y+(p-1)C)=K_{Y'}+(p-1)C_{Y'}+\Delta$ where $C_{Y'}$ is the birational transform of $C$ and $\Delta\geq 0$ is integral (because $X$ is assumed to be Gorenstein).  

In \cite[1.2]{reid_del_pezzo}, Reid gives the following short argument which shows that $Y'$ has the structure of a $K_{Y'}$-Mori fiber space.  Let $L\lin_\Q -(K_{Y'}+(p-1)C_{Y'}+\Delta)$ be a representative such that the pair $(Y', L)$ is klt, which exists because $L$ is nef and big.  There is an extremal ray $R$ which is negative for $K_{Y'}+L\lin_\Q -(p-1)C_{Y'}-\Delta$ by the cone theorem, because $C_{Y'}\neq 0$.  As $L$ is nef, $R$ is also $K_{Y'}$-negative.  $R$ contains some rational curve $\Gamma$, but we claim this cannot be a $(-1)$-curve and so $R$ gives a Mori fiber space structure.  For suppose $\Gamma$ is a $(-1)$-curve.  Then $\Gamma$ cannot be contracted over $Y$ by definition of the minimal resolution, and so by the projection formula $\Gamma\cdot L>0$.  But this is an integer, which contradicts $(K_{Y'}+L)\cdot \Gamma<0$.   
Thus $Y'$ has a Mori fiber structure, and so either satisfies $\rho(Y')=1$ and therefore is $\mathbb{P}^2$, or it has $\rho(Y')=2$ and is a $\mathbb{P}^1$ fibration.  
Assuming we are in the latter case, we show that the base $B$ is $\mathbb{P}^1$, in which case $Y'$ is either $\mathbb{P}^1\times\mathbb{P}^1$, or $H_d$.
As $\rho(Y')=2$, $NS(Y')$ is generated by a general fiber $F$ and any multisection.  As $-(K_{Y'}+(p-1)C+\Delta)$ is big and nef, $$0<-(K_{Y'}+(p-1)C+\Delta)\cdot F=(K_{Y'}+F+(p-1)C+\Delta)\cdot F$$
$$=\deg(K_F)+((p-1)C+\Delta)\cdot F=-2+((p-1)C+\Delta)\cdot F$$
This implies that $(p-1)C+\Delta$ contains at most one horizontal component, and that has degree $1$ over the base.
Suppose $(p-1)C+\Delta=D+V$ where $D$ is integral of degree $1$ over the base and $V$ is vertical.
Then $0\leq(K_X+D+V)\cdot D=K_{\tilde{D}} +V\cdot D$.  This implies $\deg(K_{\tilde{D}})<0$, which implies that $\tilde{D}$ (the normalization of $D$), and therefore the base, is rational.
Suppose on the other hand that $(p-1)C+\Delta=V$ is vertical.  As $\rho(Y'/B)=1$ ,we can replace $V$ up to numerical equivalence by a $\mathbb{Q}$-divisor to assume that $(Y',V)$ is klt.  For $\epsilon$ sufficiently small, $-(K_{Y'}+(1-\epsilon V))$ is ample, and so there is a horizontal $(K_{Y'}+(1-\epsilon)V)$-negative extremal ray.  By the cone theorem this must contain a rational curve.

The presence of $p$ further limits the possibilities, which we now work out.  We find all possibilities for effective Cartier divisors $C_{Y'}\neq 0$ and $\Delta$ on $Y'$ which are consistent with the above construction.

If $Y'=\PP^2$ then  $(p, Y, C)$ is one of the following, where $L$ is a  line:

\begin{itemize}
\item{($2$, $\PP^2$, $L$)}
\item{($2$, $\PP^2$, $C\in|2L|$)}
\item{($3$, $\PP^2$, $L$)}
\end{itemize}

For $Y'=\mathbb{P}^1\times\mathbb{P}^1$, if we label fibers of the two projections by $F_1$ and $F_2$, the possibilities for $(p,Y,C)$ are:
\begin{itemize}
\item $(2,\mathbb{P}^1\times\mathbb{P}^1, C\in |F_1+F_2|)$
\item $(2,\mathbb{P}^1\times\mathbb{P}^1, F_i)$
\end{itemize}

Lastly suppose $Y'=H_d$.  $\Pic(H_d)$ is generated by the exceptional section $D$ and fiber $F$, which satisfy $D^2=-d$, $D\cdot F=1$ and $F^2=0$.  As $\Pic(H_d)$ is 2-dimensional, there are exactly $2$ extremal rays corresponding to $H_d\to S_d$ and the Mori fiber contraction.  Thus $Y$ is either $H_d$ or $S_d$.  A divisor on $H_d$ is nef if and only if its class can be written as $AD+BF$ where $A\geq 0$ and $B\geq Ad$. In addition, if we want $AD+BF$ to be big, we need $A>0$ (by intersecting with a fiber).  We may write $K_{Y'}= -2D-(d+2)F$ and suppose $C_{Y'}\in |aD+bF|$ and $\Delta=a'D$.  Note that if $a'>0$ then $D$ is contracted over $Y$ and so $Y$ is $S_d$.  Our task is now to find all $a, b, a'$ and $p$ which satisfy:
 
\begin{equation}
\label{eq:one}
2-(p-1)a-a'>0 
\end{equation}
\begin{equation}
\label{eq:two}
d+2-(p-1)b\geq d(2-(p-1)a-a')
\end{equation}
Equality holds in \autoref{eq:two} if and only if $(K_{Y'}+(p-1)C_{Y'}+\Delta)\cdot D=0$, which happens if and only if $Y=S_d$.  Also, if $d=1$ then we must have $Y=H_1$, as $S_1\cong\mathbb{P}^2$ which is its own minimal resolution.

The restriction $2-(p-1)a-a'>0$ allows us to split into the following three cases:

\begin{enumerate}
 \item When  $a=a'=0$, the second inequality becomes $2-(p-1)b\geq d$.  The only way to satisfy this (with $b\neq 0$) is by setting $(p,a,b,a',d)=(2,0,1,0,1)$.  But these numbers give equality in the second inequality and $d=1$, and this combination is not allowed as discussed above. 

\item When $a=1$, $a'=0$, $p=2$, the second inequality becomes $2-b\geq 0$ and so $b\in\{0,1,2\}$.

\begin{enumerate}
 \item 
If $b=2$, equality holds and so  $Y=S_d$.  For $d>2$, $S_d$ is not canonical and so in these cases we would have $a'>0$. Therefore the only case which arises here is $S_2$.  This gives:

\begin{itemize}
\item $(2, S_2,C)$ where $C_{Y'}\in |D+2F|$.
\end{itemize}  

\item When $b=1$ or $0$ we get:

\begin{itemize}
\item $(2,H_d,C)$ for $d\geq 1$ where $C\in |D+F|$.
\item $(2,H_d,D)$ for $d\geq 1$.
\end{itemize}
\end{enumerate}

\item When $a=0$, $a'=1$, $D$ is contracted over $Y$, because $a'=1$.  The second inequality gives $(p-1)b\leq 2$, but here we must have equality. This gives the possibilities:
\begin{itemize}
\item $(3,S_d, F)$ for $d\geq 2$.
\item $(2,S_d, 2F)$ for $d\geq 2$.
\end{itemize}
\end{enumerate}

\end{proof}

Finally we come to the corollaries which are immediate consequences of our other results:

\begin{proof}[Proof of \autoref{cor:generic_smoothness}]
As stated in the introduction, this follows immediately from Corollary \ref{cor:normality} via the argument in the proof of \cite[5.1]{hirokado}.  Note that while Hirokado classifies the singularities which occur on the general fibre of a fibration of relative dimension $2$ under the additional assumption that the base is a curve, Schr\"oer extends this to arbitrary base in \cite{schroer}.  In particular, it is shown in \cite{schroer} that the geometric generic fiber has only rational double points as singularities and that these satisfy the same restrictions on their dual graphs as stated in \cite{hirokado} and used in the proof of \cite[5.1]{hirokado}.
\end{proof}

\begin{proof}[Proof of \autoref{cor:separably_rationally_connected}]
This follows from applying \cite[4.8]{rational_connectedness_GFR} to a resolution of the base change to the algebraic closure of $k$.
\end{proof}

\begin{proof}[Proof of \autoref{cor:quasi-elliptic}]
This is a special case of \autoref{cor:general_dimension} because a curve is smooth if and only if it is geometrically normal, and a smooth curve with $K_C\num 0$ has genus $1$.
\end{proof}

\section{Kodaira vanishing}\label{sec:irregularity}

First we prove \autoref{thm:generic_kodaira}, which implies Kodaira vanishing on Gorenstein del Pezzo surfaces of characteristic $p\geq 5$.  This proof does not rely on our earlier results.

\begin{proof}[Proof of \autoref{thm:generic_kodaira}]

Suppose that $H^1(X,\sO_X(-L))\neq 0$.  We derive a bound on $p$ depending on the dimension of $X$.

By Grothendieck duality $H^1(X,\sO_X(-mL))=H^{-1}(X,\omega_{X/k}^\bullet \otimes \sO_X(mL))$. As $X$ is $S_2$, using \cite[Prop 3.3.(6)]{patakfalvi_base_change} we obtain that the cohomology sheaves $\sH^i(\omega_{X/k}^\bullet)$ vanish for $i>-2$. Hence,  Serre vanishing yields that $H^{-1}(X,\omega_{X/k}^\bullet \otimes \sO_X(mL))=0$ for $m$ sufficiently large, which in turn implies that $H^1(X,\sO_X(-mL))=0$.  As we assume that $H^1(X,\sO_X(-L))\neq 0$, this yields some $m$ such that the Frobenius action $\Fr^*:H^1(X,\sO_X(-mL))\to H^1(X,\sO_X(-pmL))$ on the fpqc sheaves has a non trivial kernel.  Given this, \autoref{thm:ekedahl_cover} produces a purely inseparable cover $f:Z\to X$ of degree $p$ from a $G_1$, $S_2$  variety $Z$ satisfying
$$K_Z=f^*(K_X-m(p-1)L).$$

Let $\pi:Z^\nu\to Z$ be the normalization of $Z$. 
Let $C_1$ be the conductor on $Z^{\nu}$. 
 Then $K_{Z^\nu}+C_1\lin\pi^*K_Z$ 
 by \autoref{lem:two_conductors} and using that $Z$ is $G_1$ and $S_2$. 
 Let $Y$ be the normalization of an irreducible component of the reduced locus of $Z^\nu\otimes_{k}\bar{k}$, with morphism $g:Y\to Z^\nu$. According to \autoref{thm:generic}, there is an effective Weil divisor $C_2$ such that 
$$K_{Y}+C_2\lin g^*K_{Z^\nu}.$$
Letting $\Delta=g^*C_1+C_2$ and  $g\circ \pi\circ f :=h$, this gives:
\begin{equation}
\label{eq:generic_kodaira:Cartier}
K_Y+\Delta\lin h^*(K_X-m(p-1)L), 
\end{equation}
where $Y$ is a normal variety over the algebraically closed field $\bar{k}$, and $K_Y+\Delta$ is a $\bQ$-Cartier $\bZ$-divisor.
Let $C$ be  a general complete intersection curve on $Y$. 
By its generality, $C$ is contained in the smooth locus of $Y$, and is not contained in $\Supp\Delta$.  We may fix a point $c\in C$ which is not contained in $\Supp\Delta$.  By bend and break \cite[p 143, Thm 5.8]{kollar_rational} $C$ can be deformed to find an irreducible curve $\Gamma$ containing $c$ satisfying
\begin{multline*}
m(p-1)h^*L\cdot \Gamma \leq -h^*(K_X-m(p-1)L)\cdot\Gamma = -(K_Y+\Delta)\cdot\Gamma \\
\leq 2\dim(X) \frac{-(K_Y+\Delta)\cdot C}{-K_Y\cdot C}\leq 2\dim(X)\frac{-K_Y\cdot C}{-K_Y\cdot C}=2\dim(X) .
\end{multline*}
But $L$ and hence also $h^*L$ are ample Cartier divisors, so  $h^*L\cdot\Gamma$ is a positive integer.  Thus the inequality implies that $m(p-1)\leq 2\dim(X)$ (resp.  $m(p-1)< 2\dim(X)$ if $-K_X$ is ample). That is, $p\leq \frac{2\dim(X)}{m}+1\leq 2\dim(X)+1$ (resp.  $p< 2\dim(X)+1$ if $-K_X$ is ample). 

\end{proof}

In the next two propositions, we prove slightly more general versions of \autoref{cor:kodaira_gorenstein}.

\begin{proposition}\label{thm:kodaira_technical}
	Let $f : X \to T$ be a morphism from a projective, normal and Gorenstein $3$-fold with isolated singularities, to a smooth curve over a perfect field $k$  of characteristic $p\geq 5$, such that $f_* \sO_X \cong \sO_T$, $-K_X$ is $f$-ample,  and for all closed points $t \in T$:
	\begin{enumerate}
		\item $X_t$ is irreducible,
		\item $\sO_X\left( \left(-X_t\right)_{\red} \right)$ is Cartier. 
	\end{enumerate}
	If   $L$ is an ample Cartier divisor on $X$, then $H^2(X,\sO_X(K_X+L))=0$. If furthermore $p \geq 11$, then as $k$-vector spaces:
	$$H^1(X,\sO_X(K_X+L))=\bigoplus_{\parbox{75pt}{\begin{center}\tiny $t \in T$ closed point, \\[1pt]
$\left(X_t \right)_{\red}$ is non-normal, \\[1pt]
$H^1\left( \left. X_t, - L \right|_{X_t} \right) \neq 0$\end{center}}} \left(R^1 f_* \sO_X (K_X +L)\right)_t,$$
where all the summands are Artinain and non-zero.  
\end{proposition}

\begin{proof}[Proof of \autoref{thm:kodaira_technical}]
We use the assumptions on $p \geq 5$ in Step 1,  Step 3 and Step 4 and the assumption $p \geq 11$ only in Step 6 and Step 7. As the $H^2(X, K_X +L)=0$ statement is shown in Step 4, our assumptions on $p$ are valid. 

	{\bf Step 0:} \emph{ For each $i$ and each quasi-coherent sheaf $\sF$ on $X$ over $T$ we have an exact sequence:}
	\begin{equation}
	\label{eq:Leray_degenerates}
	\xymatrix{
		0 \ar[r] & H^{1}(T,R^{i-1} f_* \sF)  \ar[r] & H^i(X, \sF) \ar[r] & H^0(T,R^i f_* \sF) \ar[r] & 0.
	}
	\end{equation}
	Indeed, let $E_{\tilde{p},q}^2=R^{\tilde{p}}\Gamma(R^qf_*\sF)$ be the $({\tilde{p}},q)$ entry of the Leray spectral sequence. 
	As $T$ is a curve,  $E^2_{{\tilde{p}},0}=0$ for ${\tilde{p}}\geq 2$.  
	However, the differential $d^r_{\tilde{p},q}$ (on the $r$-th page, for $r \geq 2$) goes $E^r_{\tilde{p},q} \to E^{\tilde{p}+r, q +r-1}$. Hence, using that it exists only for $r \geq 2$, $d^r_{\tilde{p},q}=0$ always, so the spectral sequence degenerates on the $E^2$-page.

	{\bf Step 1.} \emph{Fix a closed point $t \in T$. Set $G:=(X_t)_{\red}$. In this step we show that for any integer $j>0$: $H^0\left(jG, -L|_{jG}\right)=0$, where $jG$ denotes the natural closed subscheme associated to the divisor $jG$. At the same time, we also show that if $G$ is normal, then $H^1\left(jG, -L|_{jG}\right)=0$ also holds.}
	
By our  assumption, $G$ is Cartier. Hence, for each integer $j \geq 2$  we have an exact sequence
\begin{equation*}
\xymatrix{
0 \ar[r] & \sO_G(-(j-1)G -L) \ar[r] & \sO_{jG}(-L) \ar[r] & \sO_{(j-1)G}(-L) \ar[r] & 0, 
}
\end{equation*}
inducing another exact sequence for $i=0$ or $1$:
\begin{equation*}
\xymatrix{
H^i(G, \sO_G(-(j-1)G -L)) \ar[r] & H^i\left(G,\sO_{jG}(-L)\right) \ar[r] & H^i\left(G, \sO_{(j-1)G}(-L)\right). 
}
\end{equation*}
Hence, to conclude this step, by applying the latter exact sequence inductively, it is enough to show that $H^i(G, \sO_G( -jG -L))=0$ for all $j\geq 0$ and for $i=0,1$, where in the $i=1$ case we also assume that $G$ is normal. In the next paragraph we prove the this latter claim:

As $G$ is irreducible, $G|_G \equiv 0$, and  hence $jG +L|_G$ is ample for $j \geq 0$. Hence, we obtain the required vanishing for $i=0$ by the properites of ample divisors on reduced varieties, and for $i=1$ from \autoref{thm:generic_kodaira}.

{\bf Step 2.} \emph{In this step we show that  $R^2f_*\sO_X(K_X+L)=0$.}

Fix a closed point $t \in T$. As $X_t$ is irreducible and Gorenstein, $X_t=iG$ for some integer $i>0$. Hence, Step 1 tells us that $H^0\left(X_t, -L|_{X_t}\right)=0$. So, by
Serre duality 
\begin{equation*}
0=H^2\left(X_t, \sO_{X_t}\left(K_{X_t} +L|_{X_t} \right) \right)=H^2\left(X_t, \sO_{X_t}\left( K_X +L|_{X_t} \right) \right).
\end{equation*}
Finally,  cohomology and base change concludes this step.	
	
	{\bf Step 3:}  \emph{Here we show that $R^1f_*\sO_X(K_X+L)$ is supported exactly on the finitely many closed points for which $(X_t)_{\red}$ is not normal and $H^1(X_t, K_{X_t} + L|_{X_t}) \neq 0$.}
		
Cohomology and base-change and step 2 tells us that the base-change map 
\begin{equation*}
k(t) \otimes R^1f_*\sO_X(K_X+L) \to H^1(X_t, K_{X_t} + L|_{X_t})
\end{equation*}
is an isomorphism. Then, Step 1 tells us that the latter group is zero whenever $(X_t)_{\red}$ is normal. Furthermore,  \autoref{cor:normality_short} tells us that the general fibers of $f$ are normal.

{\bf Step 4:} \emph{We conclude the proof of $H^2(X, K_X +L)=0$}:
	 
According to \autoref{eq:Leray_degenerates}, it is enough to show that  $H^0\left(T,R^2f_*\sO_X(K_X+L)\right) =H^1\left(T,R^1f_*\sO_X(K_X+L)\right)=0$. The former follows from Step 2, and the latter from Step 3.

{\bf Step 5.} { \emph Reduction to $ H^{1}(T, f_* \sO_X(K_X+L)) =0$.}
	
According to \autoref{eq:Leray_degenerates}, if we showed $ H^{1}(T, f_* \sO_X(K_X+L)) =0$, then we would have
\begin{equation*}
H^1(X, K_X +L) \cong H^0\left(T,R^1 f_* \sO_X(K_X +L)\right)  \cong
\bigoplus_{\parbox{75pt}{\begin{center}\tiny $t \in T$ closed point, \\[1pt]
$\left(X_t \right)_{\red}$ is non-normal, \\[1pt]
$H^1\left( \left. X_t, - L \right|_{X_t} \right) \neq 0$\end{center}}} \left(R^1 f_* \sO_X (K_X +L)\right)_t,
\end{equation*}
where we used Step 3 for the last isomorphism, and where all $\left(R^1 f_* \sO_X (K_X +L)\right)_t$ are Artinian and non-zero.

{\bf Step 6.}  \emph{Here we prove that $f_* \sO_X(K_{X/T} + L)$ is an ample vector bundle.} 

For this step we may assume that $k$ is algebraically closed by base-extending to $\overline{k}$ as ampleness (and all the other properties we have) are invariant under separable base-extensions.

We also note that $f_* \sO_X(K_{X/T} + L)$ is torsion-free as it is the pushforward of a torsion-free sheaf. As we work over a smooth curve this means that $f_* \sO_X(K_{X/T} + L)$ is indeed a vector bundle. Also, from this step we use the assumption $p \geq 11$, which means by \autoref{cor:generic_smoothness} that the general fibers of $f$ are smooth. 

Below we prove that $f_* \sO_X(K_{X/T} + L)$ is ample. The main idea is to use \cite[Proposition 3.6]{patakfalvi_semipositivity}. Using this proposition directly yields only nefness of $f_* \sO_X(K_{X/T} + L)$. To turn this into a proof of ampleness one needs to perform a little perturbation, which boils down eventually to the following base-change:

Choose an effective ample $\Q$-divisor $A$ on $T$ such that
\begin{itemize}
	\item $N:=L-f^*A$ is ample, 
	\item the denominators of $A$ are not divisible by $p$, and 
	\item $\Supp A$ does not contain points over which $f$ has singular fibers. 
\end{itemize}
Let $\pi:S\to T$ be a Galois cover by a smooth projective curve such that $\pi^*A$ is a $\mathbb{Z}$-divisor, and  $\pi$ is not ramified over the closed points $t \in T$ for which $X_t$ is singular. Define $Y:=X\times_T S$, and let $g:Y\to S$ and $\tau:Y\to X$ be the natural morphisms.

We claim that we may apply  \cite[Proposition 3.6]{patakfalvi_semipositivity} to $g$ and $\tau^* N$. Indeed:

\begin{itemize}
\item $\tau^* N$ is Cartier as so is $\pi^* A$.

 \item $\tau^*N $ is nef and $g$-ample as  as so is $N$. 
\item As $f$ is relatively Gorenstein, so is $g$, and hence $Y$ is Gorenstein too. Furthermore, then $K_Y= K_{Y/S} + g^* K_S = \tau^* K_{X/T} + g^* K_S$. In particular, $-K_Y$ is ample$/S$, and the index of $Y$ (which is $1$) is prime-to-$p$.
	\item $Y$ is normal, because over the non-normal fibers of $f$,  $\tau$ is \'etale.

	\item  The general fiber  of $g$ is a smooth del Pezzo surface, so  it is globally $F$-split by \cite[Example 5.5]{hara_characterisation} as the characteristic is greater than $7$.
	\item $H^0(Y_s,\sO(K_{Y/S}+ \tau^* N)|_{Y_s})=S^0(Y_s,\sO(K_{Y/S}+\tau^* N)|_{Y_s})$ for  general $s\in S$, using that the general fiber is globally $F$-split. 
	\item $g$ is flat as $S$ is a smooth curve
	\item \cite[Proposition 3.6]{patakfalvi_semipositivity} is unfortunately stated requiring that $R^ig_*\sO(K_{Y/S}+ \tau^* N)=0$ for $i>0$. However, in the proof of \cite[Proposition 3.6]{patakfalvi_semipositivity} only the following two weaker conditions are used. First, to obtain the necessary base-change properties at fibers over general points, by cohomology and base-change \cite[Cor III.12.9]{hartshorne_ag}, it is enough to chose in the proof the general point from the locus over which all degrees of cohomologies are constant. 
	Second, to guarantee the connection between fiber powers and the original family, it should be true that if $\sM$ is a line bundle on $Y$,  $g^{(n)}: Y^{(n)} \to S$ is the morphism from the $n$-times fiber product with $\sM^{(n)}$ being the line bundle induced on $Y^{(n)}$,  then $g^{(n)}_* \sM^{(n)} \cong \bigotimes_n g_* \sM$ should hold. However, it turns out that for this statement no vanishing is needed, flatness is enough as shown in \cite[Lem 3.6]{kovacs_patakfalvi}.
\end{itemize}
Summarizing, \cite[Proposition 3.6]{patakfalvi_semipositivity} applies, and it yields that $g_* \sO_Y(K_{Y/S} + \tau^* N)$ is nef. 
Therefore, $g_* \sO_Y(K_{Y/S} + \tau^* L) \cong \sO_S(\tau^* A) \otimes g_* \sO_Y(K_{Y/S} + \tau^* N)$ is ample. However,  flat base-change (\cite[III.9.3]{hartshorne_ag}) tells us that
\begin{equation*}
g_* \sO_Y(K_{Y/S} + \tau^* L) \cong \tau^* f_* \sO_X (K_{X/T} + L).
\end{equation*}
This then implies that  $f_* \sO_X (K_{X/T} + L)$ is ample too using  \cite[6.1.2-6.1.8]{lazarsfeld2}.

{\bf Step 7.} \emph{Concluding the proof.} 

According to Step 5, we only need to show that $H^1(T,f_*\sO_X(K_X+L))=0$:
\begin{multline*}
H^1(T, f_* (\omega_{X} \otimes \sO_X(L)))
\cong
\underbrace{H^1(T, \omega_T \otimes  f_* (\omega_{X/T} \otimes \sO_X(L)))}_{\omega_X \cong \omega_{X/T} \otimes f^* \omega_T \textrm{, and the projection formula}}
\\ \cong
\underbrace{\Hom_T(\omega_T \otimes  f_* (\omega_{X/T} \otimes \sO_X(L)), \omega_T)}_{\textrm{Serre duality on } T}
\cong
\Hom_T( f_* (\omega_{X/T} \otimes \sO_X(L)), \sO_T)
=
\underbrace{0}_{\textrm{Step 2}}
\end{multline*}

\end{proof}

\begin{proof}[Proof of \autoref{cor:kodaira_gorenstein}]
	The assumptions of \autoref{thm:kodaira_technical} are implied by those of \autoref{cor:kodaira_gorenstein}, and so the conclusions follow.
\end{proof}

\bibliographystyle{acm}
\bibliography{includeNice}

\end{document}